\documentclass[11pt]{article}
    \usepackage{url}
    \usepackage{verbatim}
    \usepackage[titletoc]{appendix}
    \usepackage{graphicx}
	\usepackage{amsmath}
	\usepackage{amsthm}
	\usepackage{amsfonts}
	\usepackage{graphicx}
    \usepackage{nicefrac}
    \usepackage{longtable}
    \usepackage{color}
    \usepackage{graphicx, amssymb, subcaption,graphics}

    \newcommand{\nrm}[1]{\left\| #1 \right\|}

     \newcommand{\msfT}{\mathsf{T}}

    \newcommand\dt {{\Delta t}}  
    
    \newcommand{\qref}[1]{(\ref{#1})} 
    \newcommand{\iprd}[2]{\left( #1 , #2 \right)}

	\newtheorem{thm}{Theorem}[section]
	\newtheorem{prop}[thm]{Proposition}
	\newtheorem{cor}[thm]{Corollary}
	
	\newtheorem{lem}[thm]{Lemma}
	\newtheorem{rem}[thm]{Remark}

\begin{document}
	\title{An energy stable fourth order finite difference scheme for the Cahn-Hilliard equation}
	
	\author{
Kelong Cheng \thanks{School of Science, Southwest University of Science and Technology,  Mianyang, Sichuan 621010, P. R. China (zhengkelong@swust.edu.cn )}
\and
Wenqiang Feng\thanks{Department of Mathematics, The University of Tennessee, Knoxville, TN 37996 (wfeng1@utk.edu)}
\and		
Cheng Wang\thanks{Department of Mathematics, The University of Massachusetts, North Dartmouth, MA  02747 (Corresponding Author: cwang1@umassd.edu)}	
\and
Steven M. Wise\thanks{Department of Mathematics, The University of Tennessee, Knoxville, TN 37996 (swise1@utk.edu)} 	
}

	\maketitle
	\numberwithin{equation}{section}

	\begin{abstract}
In this paper we propose and analyze an energy stable numerical scheme for the Cahn-Hilliard equation, with second order accuracy in time and the fourth order finite difference approximation in space. In particular, the truncation error for the long stencil fourth order finite difference approximation, over a uniform numerical grid with a periodic boundary condition, is analyzed, via the help of discrete Fourier analysis instead of the the standard Taylor expansion. This in turn results in a reduced  regularity requirement for the test function. In the temporal approximation, we apply a second order BDF stencil, combined with a second order extrapolation formula applied to the concave diffusion term, as well as a second order artificial Douglas-Dupont regularization term, for the sake of energy stability. As a result, the unique solvability, energy stability are established for the proposed numerical scheme, and an optimal rate convergence analysis is derived in the $\ell^\infty (0,T; \ell^2) \cap \ell^2 (0,T; H_h^2)$ norm. A few numerical experiments are presented, which confirm the robustness and accuracy of the proposed scheme. 
	\end{abstract}
	
	
\medskip
	

	\section{Introduction}
In this article we consider an energy stable scheme for the Cahn-Hilliard equation, with second order temporal accuracy and long stencil fourth order finite difference spatial approximation. For any $\phi \in H^1 (\Omega)$, with $\Omega \subset R^d$ ($d=2$ or $d=3$), the energy functional is given by (see \cite{cahn58} for a detailed derivation): 
	\begin{equation}
	\label{AC energy}
E(\phi)=\int_{\Omega}\left(\frac{1}{4} \phi^4-\frac{1}{2} \phi^2 + \frac14 + \frac{\varepsilon^2}{2}|\nabla \phi|^2\right) d {\bf x} , 
	\end{equation} 
where the parmeter $\varepsilon$ controls the diffuse interface width.  In turn, the Cahn-Hilliard equation is realized as the $H^{-1}$ conserved gradient flow of the energy functional (\ref{AC energy}): 
	\begin{equation}
	\label{CH equation}
\phi_t=\Delta \mu , \quad   \mu := \delta_\phi E = \phi^3 - \phi - \varepsilon^2 \Delta \phi , 
	\end{equation}
where $\mu$ is the chemical potential and periodic boundary conditions are assumed. Subsequently, the energy dissipation law follows from an inner product with~\eqref{CH equation} by $\mu$: $ E' (t)=-\int_{\Omega} |\nabla \mu |^2 d {\bf x} \le 0$. Meanwhile, because it is constructed as an $H^{-1}$ gradient flow, the equation is mass conservative:  $\int_\Omega \partial_t \phi  d{\bf x} = 0$.

The finite difference and finite element schemes to the CH equation have been extensively studied;  see the related references \cite{barrett99, diegel16, diegel17, elliott93, elliott89b, feng04, furihata01, guo16, he07, kay06, kay07, khiari07, kim03, WZG15, wise07}, et cetera. Meanwhile, it is observed that, most finite difference works in the existing literature have been focused on the standard second order centered difference approximation; among the finite element works, most computations are based on either linear or quadratic polynomial elements, since the implementation of higher order accurate elements are expected to be very complicated. On the other hand, a fourth order and even more accurate spatial discretization is highly desirable, for the sake of its ability to capture the more detailed structure with a reduced computational cost. Of course, the spectral/pseudo-spectral approximation is one choice; see the related references~\cite{cheng16a, LiD2016b}, etc. However, the spectral/pseudo-spectral differentiation turns out to be a global operator in space, and this feature leads to great challenges in the numerical implementations, especially in the case of an implicit treatment of nonlinear terms; and also, spectral/pseudo-spectral differentiation itself is involved with $O (N^d \ln N)$ float point calculations, instead of $O (N^d)$ scale for the finite difference ones. 

  This article is concerned with fourth order finite difference numerical approximation to the Cahn-Hilliard equation, with a theoretically justified energy stability and convergence. Among the existing fourth order finite difference works, it is worthy of mentioning~\cite{LiJ2012a}, in which the authors considered a second order accurate in time, fourth order compact difference scheme. The error estimate was derived, while the energy stability has not been theoretically proved. Similar works could also be found in~\cite{LeeC14, LiY16, Song15}. Meanwhile, we notice that, the compact difference approximation has always been involved with an additional discrete Poisson-like operator, therefore one more Poisson solver has to be included in the computational cost. Instead, if the long stencil fourth order finite difference is used in the numerical scheme, such an additional Poisson solver could be saved. In addition, the truncation error for the fourth order finite difference approximation, over a uniform numerical grid with a periodic boundary condition, is analyzed in this article. Instead of the classical maximum norm estimate for the truncation error, based on the standard Taylor expansion for the test function, we provide a discrete $\ell^2$ estimate. As a result, the regularity requirement is reduced to an $H^6$ bound for the test function, compared with the $C^6$ bound in the classical approach. In the 1-D case, such an estimate could be derived with an application of Taylor expansion in the integral form. In the 2-D and 3-D cases, the discrete and continuous Fourier expansions for the test function and its higher order derivatives are applied to obtain the discrete $\ell^2$ estimate for the truncation error, in which both the eigenvalue analysis and aliasing error control have to be considered. 
  
  In turn, we apply the long stencil fourth order difference discretization in space, combined with a second order accurate, energy stable temporal algorithm. There have been extensive studies on the second order (in time) energy stable numerical approach for the Cahn-Hilliard equation, based on the modified Crank-Nicolson version; see the related references~\cite{cheng16a, diegel17, diegel16, du91, guo16, han15}. As an alternate numerical approach with energy stability, a careful numerical experiment in a recent finite element work~\cite{yan17} reveals that a modified backward differentiation formula (BDF) method could also preserve the desired energy stability for the Cahn-Hilliard flow. Furthermore, since the nonlinear term in the BDF method has a stronger convexity than the one in the Crank-Nicolson approach, a 20 to 25 percent improvement of the computational efficiency is generally expected. In addition, due to the long stencil operators involved in the fourth order spatial discretization, such an efficiency improvement is expected to be even more prominent. Consequently, we use the second order BDF concept to derive second order temporal accuracy, but modified so that the concave diffusion term is treated explicitly. Such an explicit treatment for the concave part of the chemical potential ensures the unique solvability of the scheme without sacrificing energy stability.  An additional term $A \dt \Delta (\phi^{k+1} - \phi^k)$ is added, which represents a second order Douglas-Dupont-type regularization, and a careful calculation shows that energy stability is guaranteed, provided a mild condition $A \ge \frac{1}{16}$ is enforced. As a result of this energy stability, a uniform in time $H^1$ bound for the numerical solution becomes available, with the discrete $H^1$ norm defined at an appropriate discrete level. 
  
    With such an $H_h^1$ bound at hand, we are able to derive a discrete $\ell^6$ bound for the numerical solution, uniform in time, with the help of discrete Sobolev embedding. Such an embedding analysis could be derived from a straightforward calculation; instead, a discrete Fourier analysis, combined with certain aliasing error estimate have to involved in the derivation. In turn, such an $\ell^6$ estimate enables one to obtain an optimal rate ($O (\dt^2 + h^4)$) convergence analysis for the proposed numerical scheme, in the $\ell^\infty (0,T; \ell^2) \cap \ell^2 (0,T; H_h^2)$ norm. 
 	
	The outline of the paper is given as follows. In Section~\ref{sec:difference operator} we provide a discrete $\ell^2$ truncation error estimate for the long stencil fourth order finite difference approximation over a uniform numerical grid. The fully discrete scheme is formulated in Section~\ref{sec:numerical scheme}, with the main theoretical results stated. The proof of these results is given by Section~\ref{sec:proof}. The main numerical results are presented in Section~\ref{sec:numerical results}.  Finally, some concluding remarks are made in Section~\ref{sec:conclusion}.

	\section{The long stencil difference operator and the local truncation error estimate}\label{sec:difference operator}
	
   The long stencil fourth order finite difference formula can be derived by the Taylor expansion for the test function. In more detail, the fourth order approximations to the first and second order derivatives, over a uniform numerical grid, are given by 
	\begin{eqnarray} 
{\cal D}_{x,(4)}^1 f_i &=& \tilde{D}_x ( 1 - \frac{h^2}{6} D_x^2 ) f_i 
  = \frac{ f_{i-2} - 8 f_{i-1} + 8 f_{i+1} - f_{i+2} }{12 h} 
  = f' (x_i) + O (h^4) ,
	\label{FD-4th-1-1} 
	\\
{\cal D}_{x,(4)}^2 f_i &=& D_x^2 ( 1 - \frac{h^2}{12} D_x^2 ) f_i  = \frac{ - f_{i-2} + 16 f_{i-1} - 30 f_i + 16 f_{i+1} - f_{i+2} }{12 h^2 }
	\nonumber 
	\\
&=& f'' (x_i) + O (h^4) ,
	\label{FD-4th-1-2} 
	\end{eqnarray} 
where $\tilde{D}_x$ and $D_x^2$ are the standard centered difference approximation to the first and second order derivatives, respectively. See the detailed derivations in the related references \cite{Fornberg1988, Fornberg1998, Iserles1996, BO1978}, etc. These long stencil fourth order finite difference approximations have been extensively applied to different types of partial differential equations (PDEs), such as incompressible Boussinesq equation \cite{LWJ2003, WLJ2004}, three-dimensional geophysical fluid models \cite{LW2008, STWW2007}, the Maxwell equation \cite{FWWY2008}.

  In the consistency analysis for these long stencil fourth order finite difference schemes, the classical truncation error estimate implies that 
\begin{eqnarray} 
  &&
  \nrm{ \tau_1 }_\infty \le C h^4 \nrm{ f }_{C^5} ,  \quad 
  \mbox{with} \, \, \, (\tau_1)_i = {\cal D}_{x,(4)}^1 f_i - f' (x_i) ,   
  \label{truncation-4th-1} 
\\
  &&
   \nrm{ \tau_2 }_\infty \le C h^4 \nrm{ f }_{C^6} ,  \quad 
  \mbox{with} \, \, \, (\tau_2)_i = {\cal D}_{x,(4)}^2 f_i - f'' (x_i) ,   
  \label{truncation-4th-2} 
\end{eqnarray} 
in which the $C^m$ regularity of the test function is involved. Meanwhile, for most time-dependent PDEs, a max-norm bound of the truncation error (in the $\nrm {\, \cdot\, }_\infty$ norm) is not necessary in the numerical analysis.  Indeed, a discrete $\ell^2$ bound of the truncation error is typically sufficient in the convergence analysis. Subsequently, a natural question arises: could a discrete $\ell^2$ estimate be available for the truncation error associated with the fourth order finite difference approximation, which only requires an $H^m$ regularity for the test function? 
	
  This important issue is studied in this section. We begin with the analysis for the 1-D case. 
  
\subsection{A discrete $\ell^2$ truncation error estimate over a 1-D numerical grid}

Consider a 1-D domain $\Omega = (0, L)$, a uniform grid $x_i=i h$, 
with $h = L/N^*$, $0 \le i \le N^* -1$. For any discrete grid function $g$, which is evaluated at grid points $x_i$, $0 \le i \le N^* -1$, the discrete $\ell^2$ norm is introduced as 
\begin{equation} 
  \nrm{g}_2^2 =  h \sum_{i=0}^{N^* -1} g_i^2 .  \label{def-1d-1}
\end{equation}  

\begin{prop} \label{thm-1d}
   For $f \in H^6_{per} (\Omega) := \left\{ f \in H^6 (\Omega): \mbox{f is periodic} \right\}$, we have 
\begin{equation} 
  \nrm{ \tau }_2  \le C h^4 \nrm{f}_{H^6} ,  \quad 
  \mbox{with} \, \, \, \tau_i = {\cal D}_{x,(4)}^2 f_i - f'' (x_i) , \label{est-1d-1}
\end{equation} 
in which the fourth order finite difference operator 
${\cal D}_{x,(4)}^2$ is given by \qref{truncation-4th-2}. 
\end{prop}

\begin{proof} 
   An application of Taylor's series in integral form shows that 
\begin{eqnarray} 
  f_{i+1} &=& f_i + h f' (x_i) +\frac{h^2}{2} f'' (x_i) 
    + \frac{h^3}{6} f''' (x_i) +  \frac{h^4}{24} f^{(4)} (x_i)  
    + \frac{h^5}{120} f^{(5)} (x_i)   \nonumber 
\\
  && 
   +  \int_{x_i}^{x_{i+1}} \frac{f^{(6)} (t)}{120} (x_{i+1} -t)^5  dt ,  
   \label{FD-4th-Taylor-I-1-1} 
\\
  f_{i-1} &=& f_i - h f' (x_i) +\frac{h^2}{2} f'' (x_i) 
   - \frac{h^3}{6} f''' (x_i) +  \frac{h^4}{24} f^{(4)} (x_i)  
   - \frac{h^5}{120} f^{(5)} (x_i)   \nonumber 
\\
  && 
    +  \int_{x_i}^{x_{i-1}} \frac{f^{(6)} (t)}{120} (x_{i-1} -t)^5  dt  ,  
    \label{FD-4th-Taylor-I-1-2}
\\
  f_{i+2} &=& f_i + 2 h f' (x_i) +\frac{4 h^2}{2} f'' (x_i) 
    + \frac{8 h^3}{6} f''' (x_i) +  \frac{16 h^4}{24} f^{(4)} (x_i)  
   + \frac{32 h^5}{120} f^{(5)} (x_i)   \nonumber 
\\
  && 
   +  \int_{x_i}^{x_{i+2}} \frac{f^{(6)} (t)}{120} (x_{i+2} -t)^5  dt ,  
   \label{FD-4th-Taylor-I-1-3}  
\\
  f_{i-2} &=& f_i - 2h f' (x_i) +\frac{4 h^2}{2} f'' (x_i) 
   - \frac{8 h^3}{6} f''' (x_i) +  \frac{16 h^4}{24} f^{(4)} (x_i)   
    - \frac{32 h^5}{120} f^{(5)} (x_i)   \nonumber 
\\
  && 
    +  \int_{x_i}^{x_{i-2}} \frac{f^{(6)} (t)}{120} (x_{i-2} -t)^5  dt  .  
     \label{FD-4th-Taylor-I-1-4}
\end{eqnarray} 
These expansions in turn yield that 
\begin{eqnarray} 
  &&
   \tau_i = {\cal D}_{x,(4)}^2 f_i - f'' (x_i) 
 = \frac{-f_{i-2} + 16 f_{i-1} - 30 f_i + 16 f_{i+1} - f_{i+2}}{12h^2}   - f'' (x_i) \nonumber 
\\
  &=& 
   \frac{1}{1440 h^2}   
   \Bigl( 16 \int_{x_i}^{x_{i+1}} f^{(6)} (t) (x_{i+1} -t)^5  dt  
  -  16 \int_{x_{i-1}}^{x_i} f^{(6)} (t)  (x_{i-1} -t)^5  dt   \nonumber 
\\
  &&
  -  \int_{x_i}^{x_{i+2}}  f^{(6)} (t)  (x_{i+2} -t)^5  dt 
  +  \int_{x_{i-2}}^{x_i}  f^{(6)} (t)  (x_{i-2} -t)^5  dt   \Bigr) .  
    \label{FD-4th-Taylor-I-3}
\end{eqnarray}
Meanwhile, an application of H\"older's inequality implies that 
\begin{eqnarray} 
  &&
   \left| \int_{x_i}^{x_{i+1}} f^{(6)} (t) (x_{i+1} -t)^5  dt  \right|  \nonumber 
\\
   &\le&  \nrm{   f^{(6)} (t)  }_{\ell^2 (x_i, x_{i+1}) }  
   \cdot  \nrm{  (x_{i+1} -t)^5 }_{\ell^2 (x_i, x_{i+1}) }  
  = \frac{1}{\sqrt{11} } h^{11/2}  \nrm{   f^{(6)} (t) }_{\ell^2 (x_i, x_{i+1}) }  ,  
  \label{FD-4th-Taylor-I-4}
\\
  &&
   \left| \int_{x_{i-1}}^{x_i}  f^{(6)} (t) (x_{i-1} -t)^5  dt  \right| 
   \le \frac{1}{\sqrt{11}} h^{11/2}  \nrm{  f^{(6)} (t)  }_{\ell^2 (x_{i-1}, x_i) }  ,  
  \label{FD-4th-Taylor-I-4-2}
\\
  &&
   \left| \int_{x_i}^{x_{i+2}} f^{(6)} (t) (x_{i+2} -t)^5  dt  \right| 
   \le \frac{32 \sqrt{2}}{\sqrt{11} } h^{11/2}  
   \nrm{  f^{(6)} (t)  }_{\ell^2 (x_i, x_{i+2}) }  ,  
  \label{FD-4th-Taylor-I-4-3}
\\
  &&
   \left|  \int_{x_{i-2}}^{x_i}  f^{(6)} (t) (x_{i-2} -t)^5  dt  \right|   
   \le \frac{32 \sqrt{2}}{\sqrt{11} } h^{11/2}    
   \left\|   f^{(6)} (t)  \right\|_{\ell^2 (x_{i-2}, x_i) }  .  
  \label{FD-4th-Taylor-I-4-4}
\end{eqnarray}
Its combination with \qref{FD-4th-Taylor-I-3} leads to 
\begin{eqnarray} 
  \left|   \tau_i   \right| 
  \le  \frac{1 + 2 \sqrt{2} }{45 \sqrt{11} } h^{7/2}    
   \left\|   f^{(6)} (t)  \right\|_{\ell^2 (x_{i-2}, x_{i+2}) }  .  
  \label{FD-4th-Taylor-I-5}
\end{eqnarray}
In turn, a discrete $\ell^2$ estimate of the truncation error $\tau$ is obtained: 
\begin{eqnarray} 
   &&
   h \sum_{i=0}^{N^* -1}  \left|   \tau_i   \right|^2 
  \le C h^8  \sum_{i=0}^{N^* -1}   
  \nrm{   f^{(6)} (t)  }^2_{\ell^2 (x_{i-2}, x_{i+2}) }
  \le 4 C h^8    
  \left\|   f^{(6)} (t)  \right\|^2_{\ell^2 (0, L) } ,    \nonumber 
\\
  &&  \quad  
  \mbox{i.e.}  \quad \nrm{ \tau }_2 \le C h^4  \nrm{  f  }_{H^6 (0, L) } , 
   \label{FD-4th-Taylor-I-6}
\end{eqnarray}
in which the second step comes from the following obvious fact: 
\begin{eqnarray} 
  \sum_{i=0}^{N^* -1}   
  \nrm{  f^{(6)} (t)  }^2_{\ell^2 (x_{i-2}, x_{i+2}) }
  \le 4 \nrm{  f^{(6)} (t) }^2_{\ell^2 (0, L) } .  
   \label{FD-4th-Taylor-I-7}
\end{eqnarray}
Proposition~\ref{thm-1d} is proven.  
\end{proof}

\begin{rem} 
  A detailed calculation reveals that the standard Taylor expansion results in the classical truncation error estimate as \qref{truncation-4th-2}, with a $C^6$ regularity requirement for the test function; while the Taylor expansion in the integral form gives an improved truncation error estimate \qref{est-1d-1}, in which only an $H^6$ regularity requirement is set for the test function. 
\end{rem} 

\begin{rem} 
  For the long stencil fourth order finite difference approximation \qref{FD-4th-1-1} to the first order derivative, the following discrete $\ell^2$ truncation error estimate can be derived in a similar manner: 
\begin{equation} 
   \nrm{ \tau_1 }_2 \le C h^4 \nrm{ f }_{H^5} ,  \quad 
  \mbox{with} \, \, \, (\tau_1)_i = {\cal D}_{x,(4)}^1 f_i - f' (x_i) .   
  \label{truncation-4th-1-2} 
\end{equation} 
The details are skipped for brevity. 
\end{rem}

\subsection{The 2-D and 3-D analyses}

Consider a 2-D domain $\Omega = (0, L)^2$ with a uniform grid 
$(x_i , y_j) = (i h , j h)$, $h = L/N^*$, $0 \le i , j \le 2 N$, and periodic boundary conditions in both directions. For simplicity of presentation in the Fourier analysis, it is assumed that $N^* = 2N+1$.  We also make a 2-D extension of the fourth order long stencil finite difference operator as $\Delta_{h,(4)} = {\cal D}_{x,(4)}^2 + D_{y, (4)}^2$: 
\begin{eqnarray} 
  \hspace{-0.35in}  &&
   {\cal D}_{x,(4)}^2 f_{i,j} = D_x^2 ( 1 - \frac{h^2}{12} D_x^2 ) f_{i,j} 
  = \frac{ - f_{i-2,j} + 16 f_{i-1,j} - 30 f_{i,j} + 16 f_{i+1,j} - f_{i+2,j} }{12 h^2 } , 
   \label{FD-4th-2d-1} 
\\
  \hspace{-0.35in}  && 
  {\cal D}_{y,(4)}^2 f_{i,j} = D_y^2 ( 1 - \frac{h^2}{12} D_y^2 ) f_{i,j} 
  = \frac{ - f_{i,j-2} + 16 f_{i,j-1} - 30 f_{i,j} + 16 f_{i,j+1} - f_{i,j+2} }{12 h^2 } . 
  \label{FD-4th-2d-2} 
\end{eqnarray} 
In addition, for any discrete grid function $g$, which is evaluated at 2-D grid points $(x_i, y_j)$, $0 \le i , j \le N^* -1$, the discrete $\ell^2$ norm is defined as 
\begin{equation} 
  \nrm{g}_2^2 = h^2 \sum_{i,j=0}^{N^* -1} g_{i,j}^2  .  \label{def-2d-1}
\end{equation}  
 
\begin{prop} \label{thm-2d}
  For $f \in H^6_{per} (\Omega)$, we have 
\begin{equation} 
   \nrm{ \tau }_2  \le C h^4 \nrm{f}_{H^6} ,  \quad 
  \mbox{with} \, \, \, \tau_{i,j} = \Delta_{h,(4)} f_{i,j} 
 - (\Delta f) (x_i, y_j) , \label{est-2d-1}
\end{equation} 
in which $C$ only depends on $L$. 
\end{prop} 

\subsection{Review of Fourier series and interpolation}

For $f(x,y) \in \ell^2 (\Omega)$, $\Omega= (0,L)^2$,  with Fourier series  
\begin{eqnarray} 
 f(x,y) =   \sum_{k,l=-\infty}^{\infty}  
 \hat{f}_{k,l} {\rm e}^ {2 \pi {\rm i} ( k x + l y)/L },  \quad   
  \mbox{with} \quad 
    \hat{f}_{k,l}  =  \int_{\Omega}  f(x,y)  
    {\rm e}^{-2 \pi {\rm i} ( k x + ly) /L } d x dy , 
\end{eqnarray}
its truncated series is defined as the projection onto the space 
${\cal B}^N$ of trigonometric polynomials in $x$ and $y$ of degree up to $N$,  
given by 
	\begin{equation} 
{\cal P}_N f(x,y) =   \sum_{k,l=-N}^{N} \hat{f}_{k,l} {\rm e}^ {2 \pi {\rm i} ( k x + l y) /L } . 
	\end{equation}
Meanwhile, an interpolation operator ${\cal I}_N$ is introduced if one wants an
approximation which matches the function at a given set of points. 
Given a uniform numerical grid with $(2N+1)$ points in each dimension
and a grid function ${\mathbf f}, $ where ${\mathbf f}_{i,j} = f(x_i, y_j) $,  the Fourier interpolation of the function is defined by 
	\begin{equation} 
{\cal I}_N f (x,y) =   \sum_{k,l=-N}^{N}    (\hat{f}_c^N)_{k,l} {\rm e}^ {2 \pi {\rm i} ( k x + ly) /L }, \label{interpolation} 
	\end{equation}
where the $(2N+1)^2$ pseudospectral coefficients $ (\hat{f}_c^N)_{k,l}$ %\textcolor{red}{(I changed the subscripts on the pseudospectral coefficients.  Is my change correct?)} 
are be computed based on  
the interpolation requirement
	\[
f(x_i, y_j) =   {\cal I}_N f (x_i, y_j)
	\]
at the $(2N+1)^2$ equidistant points \cite{Boyd2001, GO1977, HGG2007}.  Note that these \emph{collocation} coefficients can be efficiently computed using the fast Fourier transform (FFT). Also, observe that these interpolation coefficients are not equal to the actual Fourier coefficients.  The difference between them is known as the aliasing error. In general, ${\cal P}_N f(x,y)  \ne {\cal I}_N f(x,y) $, and even ${\cal P}_N f(x_i, y_j)  \ne {\cal I}_N f(x_i, y_j) $, except of course in the case that $ f \in {\cal B}^N$.

The consistency and accuracy of the Fourier projection and interpolation have been well established in the existing literature.  As long as $f\in H^m_{per}(\Omega)$, %\textcolor{blue}{(Is this acceptable?)} 
the convergence of the derivatives of the projection and interpolation is given by 
	\begin{eqnarray}  
 \|\partial^\alpha f  - \partial^\alpha {\cal P}_N f  \| 
  &\leq&  C \| f^{(m)} \| h^{m-|\alpha|} , \quad 
\mbox{for} \, \, \, 0 \le |\alpha| \le m ,
	\nonumber 
	\\
  \|\partial^\alpha f  - \partial^\alpha {\cal I}_N f  \| 
  &\leq&  C \| f \|_{H^m} h^{m-|\alpha|} , \quad 
\mbox{for} \, \, \, 0 \le |\alpha| \le m ,  \, m > \frac{d}{2} ,   
 	\label{spectral-approximation}
	\end{eqnarray}
where $\| \cdot \|$ denotes the standard $\ell^2$ norm and $d$ is the dimension.  For more details, see the  discussion of approximation theory by Canuto and  Quarteroni \cite{canuto82} .

\subsection{Proof of Proposition~\ref{thm-2d} } 

Assume that $f \in H^6_{per}$ has a Fourier expansion 
\begin{eqnarray} 
  f (x,y) = \sum_{k,l=-\infty}^{\infty}   
   \hat{f}_{k,l} \mbox{exp} \left( 2 \pi {\rm i} ( k x + ly )/L \right)  
    . \label{expansion-2d-f-1}
\end{eqnarray}
The Parseval equality shows that 
\begin{equation} 
  \nrm{f}_{2}^2 =  L^2 \sum_{k,l=-\infty}^\infty 
    \left| \hat{f}_{k,l} \right|^2 .  \label{est-2d-f-2}
\end{equation}
Similarly, for the derivatives, we have 
\begin{eqnarray} 
   \Delta^{m_0} f (x,y) = \sum_{k,l=-\infty}^\infty \Bigl( - ( \frac{2 \pi}{L} )^2 ( k^2 + l^2 )^2   \Bigr)^{m_0} \hat{f}_{k,l} \mbox{exp} \left( 2 \pi {\rm i} ( k x + ly )/L \right) , 
%   \quad \mbox{for $m \ge 2$, m even} ,  
   \label{expansion-2d-f-2}
\end{eqnarray}
for $m_0 \ge 1$, and the corresponding Parseval equality gives 
\begin{equation} 
  \nrm{\Delta^{m_0} f}_{2}^2 =  L^2
    \sum_{k,l=-\infty}^\infty \left(  \frac{4 \pi^2}{L^2}  ( k^2 + l^2 )
    \right)^{2 m_0}     \left| \hat{f}_{k,l} \right|^2 .  
    \label{est-2d-f-3}
\end{equation}

In particular, we see that 
\begin{eqnarray} 
  \nrm{\Delta f}_{2}^2 &=& L^2
    \sum_{k,l=-\infty}^\infty \left(  \frac{4 \pi^2}{L^2}  ( k^2 + l^2 )
    \right)^2      \left| \hat{f}_{k,l} \right|^2 , \nonumber 
\\ 
   \nrm{\Delta^3 f}_{2}^2 &=&  L^2
    \sum_{k,l=-\infty}^\infty \left(  \frac{4 \pi^2}{L^2}  ( k^2 + l^2 )
    \right)^6    \left| \hat{f}_{k,l} \right|^2 .  \label{est-2d-f-4}
\end{eqnarray}

  We also note that for any (periodic) discrete grid function over $(x_i , y_j)$, $0 \le i , j \le 2 N$, with a discrete Fourier expansion 
\begin{equation} 
  g_{i,j} = \sum_{k,l=-N}^{N}  
   \hat{g}_{k,l} \mbox{exp} \left( 2 \pi {\rm i} ( k x_i + l y_j ) /L \right)  , 
   \label{expansion-2d-g-1}
\end{equation}
the corresponding discrete Parseval equality is valid:
\begin{equation} 
  h^2 \sum_{i,j=0}^{2N} | g_{i,j} |^2 = L^2 \sum_{k,l=-N}^{N} 
   \left| \hat{g}_{k,l} \right|^2  .  \label{est-2d-g-2}
\end{equation}

  Meanwhile, we observe that the discrete Fourier expansion of $f$ over 
the uniform grid $(x_i, y_j)$, $0 \le i , j \le 2 N$, is not the projection of (\ref{expansion-2d-f-1}), due to the appearance of aliasing errors.  
A more careful calculation reveals that 
\begin{eqnarray} 
  f_{i,j} = \sum_{k,l=-N}^{N} \tilde{\hat{f}}_{k,l} 
  \mbox{exp} \left( 2 \pi {\rm i} ( k x_i + l y_j ) /L \right)  , 
   \quad \mbox{with}  \, \, \, 
   \tilde{\hat{f}}_{k,l} = \sum_{k_1, l_1 =-\infty}^{\infty} \hat{f}_{k+k_1 N^*, l + l_1 N^*} .
   \label{expansion-2d-f-dis-1}
\end{eqnarray}
(See the related derivations in \cite{tadmor86}.)  In turn, taking the centered difference ${\cal D}_{x,(4)}^2$, ${\cal D}_{y,(4)}^2$ on $f$ and making use of the fact that 
$\mbox{exp} \left( 2 \pi {\rm i} ( k x_i + l y_j ) /L \right)$ is also an eigenfunction of the discrete operator $\Delta_{h,(4)}$ lead to the following formulas:
	\begin{equation}     
   \label{expansion-2d-f-dis-2}
  \Delta_{h,(4)} f_{i,j} = 
     \sum_{k,l=-N}^{N} \left( \lambda_{kx,(4)} + \lambda_{ly,(4)} \right) 
     \tilde{\hat{f}}_{k,l} \mbox{exp} \left( 2 \pi {\rm i} ( k x_i + l y_j ) /L \right)  ,  
	\end{equation}
where
	\begin{equation} 
  \lambda_{kx,(4)} = \lambda_k - \frac{h^2}{12} \lambda_k^2 ,  \, \, 
  \lambda_{ly,(4)} = \lambda_l - \frac{h^2}{12} \lambda_l^2 , \, \, 
  \lambda_k = \frac{- 4 \mbox{sin}^2 (k \pi h/L)}{h^2} , 
    \quad \lambda_l = \frac{- 4 \mbox{sin}^2 (l \pi h/L)}{h^2} . 
	\end{equation}
Moreover, a differentiation of the Fourier expansion (\ref{expansion-2d-f-1}) leads to 
\begin{equation} 
  \Delta f (x,y) = \sum_{k,l=-\infty}^\infty 
  \left( \frac{-4 \pi^2}{L^2} (k^2+ l^2 ) \right)  
   \hat{f}_{k,l}  \mbox{exp} \left( 2 \pi {\rm i}  (k x_i + l y_j ) / L \right)   , 
   \label{expansion-2d-f-3}
\end{equation}
and its interpolation at $(x_i, y_j)$ gives 
\begin{eqnarray} 
  (\Delta f)_{i,j} = 
    \sum_{k,l=-N}^{N} \tilde{\hat{f}}_{k,l}^{(2)} \mbox{exp} \left( 2 \pi {\rm i}  (k x_i + l y_j ) / L \right), 
\end{eqnarray}
with
\begin{eqnarray} 
   \tilde{\hat{f}}_{k,l}^{(2)}    
   = \sum_{k_1,l_1=-\infty}^{\infty} \left( \frac{-4 (k + k_1 N^*) ^2\pi^2
     -4 (l + l_1 N^*)^2 \pi^2}{L^2}  \right)  \hat{f}_{k+k_1 N^*,l + l_1 N^*}  .     
   \label{expansion-2d-f-dis-3}
\end{eqnarray}
%\textcolor{red}{(Is this last equation correct? Should the coefficients $\hat{f}_{k,l}$ have indices as in the latter part of \eqref{expansion-2d-f-dis-1}?)}

  Therefore, the difference between (\ref{expansion-2d-f-dis-2}) and 
(\ref{expansion-2d-f-dis-3}) gives 
\begin{eqnarray} 
  \tau_{i,j} = 
     \sum_{k=-N}^{N} \left( 
     \left( \lambda_{kx,(4)} + \lambda_{ly,(4)} \right) \tilde{\hat{f}}_{k,l}  
     - \tilde{\hat{f}}_{k,l}^{(2)} \right)  
     \mbox{exp} \left( 2 \pi {\rm i}  (k x_i + l y_j ) / L \right)  .     
   \label{expansion-2d-f-dis-4}
\end{eqnarray} 
As a result, an application of discrete Parseval equality yields 
\begin{eqnarray} 
  \nrm{ \tau }_2^2 = L^2  \sum_{k,l=-N}^{N} \left| 
     \left( \lambda_{kx,(4)} + \lambda_{ly,(4)} \right) \tilde{\hat{f}}_{k,l}  
     - \tilde{\hat{f}}_{k,l}^{(2)} \right|^2  .   
   \label{est-2d-f-dis-2}
\end{eqnarray}  
Moreover, a detailed comparison between (\ref{expansion-2d-f-dis-1}) and 
(\ref{expansion-2d-f-dis-3}) results in 
\begin{eqnarray} 
   \left( \lambda_{kx,(4)} + \lambda_{ly,(4)} \right) \tilde{\hat{f}}_{k,l} - \tilde{\hat{f}}_{k,l}^{(2)}   
   &=& \left( \left( \lambda_{kx,(4)} + \frac{4 k^2\pi^2}{L^2} \right) 
    +  \left( \lambda_{ly,(4)} + \frac{4 l^2\pi^2}{L^2} \right)  \right) \hat{f}_{k,l}  \nonumber 
\\
  & & 
   + \sum_{ k_1,l_1=-\infty \atop (k_1,l_1) \ne (0,0)  }^{\infty} \biggl\{  \left( \lambda_{kx,(4)} 
   + \frac{4 (k+ k_1 N^*)^2 \pi^2}{L^2} \right) 
   \nonumber 
\\
  &&  
   +  \left( \lambda_{ly,(4)}   
   + \frac{4 (l+ l_1 N^*)^2 \pi^2}{L^2} \right)  \biggr\}
   \hat{f}_{k+k_1 N^* , l + l_1 N^* }  .     
   \label{est-2d-coefficient-1}
\end{eqnarray}
The estimates of the above terms are given by the following lemmas. The proofs will be provided in the appendix. 

\begin{lem}  \label{lem-1}
  We have, for some $C_1>0$,
\begin{equation} 
  \left| \lambda_{kx,(4)} + \frac{4 k^2 \pi^2}{L^2} \right| 
  \le C_1 h^4 \left( \frac{2 k \pi}{L} \right)^6 ,   \quad 
  \left| \lambda_{ly,(4)} + \frac{4 l^2\pi^2}{L^2} \right| 
  \le C_1 h^4 \left( \frac{2 l \pi}{L} \right)^6 ,  
  \quad \forall \, \, 0 \le k,l \le N .  \qquad 
    \label{est-2d-coefficient-2}
\end{equation}
%in which $C_1$ only depends on $L$.
\end{lem}

\begin{lem} \label{lem-2}
  We have 
\begin{eqnarray}
   \sum_{k,l=-N}^{N} 
   \left| \sum_{k_1,l_1=-\infty \atop (k_1,l_1) \ne (0,0)}^{\infty}  \left( \lambda_{kx,(4)} 
   + \frac{4 (k+ k_1 N^*)^2 \pi^2}{L^2} \right) 
   \hat{f}_{k+k_1 N^* , l + l_1 N^* }  \right|^2 
   \le C_2 h^8 \nrm{f}_{H^6}^2 ,  \nonumber 
\\
  \sum_{k,l=-N}^{N} 
   \left| \sum_{k_1,l_1=-\infty\atop(k_1,l_1) \ne (0,0)}^{\infty}  \left( \lambda_{ly,(4)} 
   + \frac{4 (l+ l_1 N^*)^2 \pi^2}{L^2} \right) 
   \hat{f}_{k+k_1 N^*, l + l_1 N^*}  \right|^2 
   \le C_2 h^8 \nrm{f}_{H^6}^2 ,    \label{est-2d-coefficient-3}
\end{eqnarray}
where $C_2>0$ is a constant that only depends on $L$. %\textcolor{blue}{(I changed the summation lower index above.)}
\end{lem} 

  A direct consequence of Lemma~\ref{lem-1} shows that 
\begin{eqnarray} 
   \sum_{k,l=-N}^{N}  \left( \lambda_{kx,(4)} + \frac{4 k^2\pi^2}{L^2} \right)^2
    \left| \hat{f}_{k,l} \right|^2    
     \le \sum_{k,l=-N}^{N} C_1^2 h^8 \left( \frac{2 k \pi}{L} \right)^{12}  
     \left| \hat{f}_{k,l} \right|^2 
     \le \tilde{C}_1 h^8 \nrm{f}_{H^6}^2 ,   \nonumber 
\\
  \sum_{k,l=-N}^{N}  \left( \lambda_{ly,(4)} + \frac{4 l^2\pi^2}{L^2} \right)^2
    \left| \hat{f}_{k,l} \right|^2    
     \le \sum_{k,l=-N}^{N} C_1^2 h^8 \left( \frac{2 l \pi}{L} \right)^{12}  
     \left| \hat{f}_{k,l} \right|^2 
     \le \tilde{C}_1 h^8 \nrm{f}_{H^6}^2 , \label{est-2d-f-dis-3}
\end{eqnarray}
with $\tilde{C}_1 = \frac{C_1^2}{L^2}$, 
in which we used the estimate (\ref{est-2d-f-4})
\begin{equation} 
    \sum_{k,l=-N}^{N}  
   \left( \left( \frac{2 k \pi}{L} \right)^{12}  + \left( \frac{2 l \pi}{L} \right)^{12}  \right) 
   \left| \hat{f}_{k,l} \right|^2  
    \le \sum_{k,l=-\infty}^{\infty}  
   \left( \left( \frac{2 k \pi}{L} \right)^{12}  + \left( \frac{2 l \pi}{L} \right)^{12}  \right) 
   \left| \hat{f}_{k,l} \right|^2 
    \le \frac{1}{L^2} \nrm{\Delta^3 f}^2 .  \label{est-2d-f-4-alt}
\end{equation}

  A combination of (\ref{est-2d-coefficient-1}), (\ref{est-2d-f-dis-3}) 
and Lemma \ref{lem-2} indicates that 
\begin{eqnarray} 
&& \hspace{-1in}\sum_{k,l=-N}^{N}  
    \left|  \left( \lambda_{kx,(4)} + \lambda_{ly,(4)} \right) \tilde{\hat{f}}_{k,l} 
    - \tilde{\hat{f}}_{k,l}^{(2)}  \right|^2  \nonumber 
\\
    &\le& 4 \sum_{k,l=-N}^{N}  \Biggl\{ 
     \left( \left( \lambda_{kx,(4)} + \frac{4 k^2\pi^2}{L^2} \right)^2  
   + \left( \lambda_{ly,(4)} + \frac{4 l^2\pi^2}{L^2} \right)^2  \right)
    \left| \hat{f}_{k,l} \right|^2   \nonumber 
\\
  & & 
      + \left| \sum_{k_1,l_1=-\infty \atop (k_1,l_1) \ne (0,0)}^{\infty}  \left( \lambda_{kx,(4)} 
   + \frac{4 (k+ k_1 N^*)^2 \pi^2}{L^2} \right) 
   \hat{f}_{k+k_1 N^*, l + l_1 N^*}  \right|^2   \nonumber 
\\
  & & 
   + \left| \sum_{k_1,l_1=-\infty\atop (k_1,l_1) \ne (0,0)}^{\infty}  \left( \lambda_{ly,(4)} 
   + \frac{4 (l+ l_1 N^*)^2 \pi^2}{L^2} \right) 
   \hat{f}_{k+k_1 N^*, l + l_1 N^*}  \right|^2   \Bigg\}  \nonumber  
\\
   &\le& \tilde{C}_2 h^8  \nrm{f}_{H^6}^2 , 
   \label{est-2d-f-dis-4}
\end{eqnarray}
where $\tilde{C}_2 = 8 \tilde{C}_1 + 8 C_2$. Observe that the Cauchy inequality
	\[|a_1 + a_2 + a_3 + a_4|^2 
\le 4 ( | a_1 |^2 + | a_2 |^2 + | a_3 |^2 + | a_4 |^2)
	\]
was applied in the first step. 

  Finally, a substitution of (\ref{est-2d-f-dis-4})  into (\ref{est-2d-f-dis-2}) results in (\ref{est-2d-1}), with the constant $C$ given by 
\begin{equation} 
  C = \sqrt{\tilde{C}_2 } L .  \label{def-2d-constant}
\end{equation}
This completes the proof of Proposition~\ref{thm-2d}. 

\subsection{An extension to a 3-D domain} 
  
Similarly, consider a 3-D domain $\Omega = (0, L)^3$ with a uniform grid 
$(x_i , y_j, z_k) = (i h , j h, k h)$, $h = L/N^*$, ($N^* = 2N+1$), $0 \le i , j, k \le 2 N$, and periodic boundary conditions in both $x$, $y$ and $z$ directions. The 3-D extension of the fourth order long stencil finite difference operator as $\Delta_{h,(4)} = {\cal D}_{x,(4)}^2 + D_{y, (4)}^2 + D_{z, (4)}^2$ can be defined in the same way as \qref{FD-4th-2d-1}-\qref{FD-4th-2d-2}. For any 3-D discrete grid function $g$, at grid points $(x_i, y_j, z_k)$, $0 \le i , j, k \le N^* -1$, the discrete $\ell^2$ norm is given by 
\begin{equation} 
  \nrm{g}_2 = \Bigl( h^3 \sum_{i,j=0}^{N^* -1} g_{i,j,k}^2  \Bigr)^{\frac12} .  \label{def-3d-1}
\end{equation}  

The discrete $\ell^2$ truncation error estimate can be performed in a similar fashion as in the 2-D case.  Both the continuous Fourier expansion for the test function and the discrete expansions of its higher order derivatives interpolated at the numerical grid points have to be analyzed. Lemma \ref{lem-1} is still valid, which provides a detailed eigenvalue comparison and analysis between $\Delta_{h, (4)}$ and $\Delta$. Furthermore, Lemma \ref{lem-2} can be established in a similar way, and the convergence property of the following triple series plays a key role: 
\begin{eqnarray} 
  &&  
  \sum_{k_1,l_1,m1=-\infty\atop (k_1,l_1,m1) \ne (0,0,0)}^{\infty}  
  \frac{ 1 }{\left( ( | k_1 | - \frac12)^2 + ( | l_1 | - \frac12)^2 
  + ( | m_1 | - \frac12)^2\right)^{\beta_0} } , \label{convergent-double-3}
\end{eqnarray} 
with $\beta_0 =2$ is convergent, where $\beta_0 = \frac{K_0}{2} = 2$, with the accuracy order $K_0=4$. 

As a result, the following theorem could be established. Its detailed proof is skipped for brevity. 
 
\begin{prop} \label{thm-3d}
  For $f \in H^6_{per} (\Omega)$, with $\Omega = (0,L)^3$, we have 
\begin{equation} 
   \nrm{ \tau }_2  \le C h^4 \nrm{f}_{H^6} ,  \quad 
  \mbox{with} \, \, \, \tau_{i,j,k} = \Delta_{h,(4)} f_{i,j,k} 
 - (\Delta f) (x_i, y_j, z_k) , \label{est-3d-1}
\end{equation} 
in which $C$ only depends on $L$. 
\end{prop}

\section{The numerical scheme for the Cahn-Hilliard equation} \label{sec:numerical scheme} 

For simplicity, we focus our attention on a 2-D domain. The extension to the 3-D case is expected to be straightforward. 

\subsection{The spatial discretization and the related notations}

%Let $\Omega = (0,L_x)\times(0,L_y)$, where, for simplicity, we assume $L_x =L_y =: L > 0$. 
As defined in the previous section, it is assumed that assume $\Omega = (0,L)^2$. We write $L = m\cdot h$, where $m$ is a positive integer. The parameter $h = \frac{L}{m}$ is called the mesh or grid spacing. We define the following uniform, infinite grid with grid spacing $h>0$: $E := \{ x_{i} \ |\ i\in {\mathbb{Z}}\}$, with $x_i := (i-\frac12) h$. Consider the following 2-D discrete periodic function spaces: 
	\begin{eqnarray*}
{\mathcal V}_{\rm per} &:=& \{\nu: E\times E\rightarrow {\mathbb{R}}\ | \ \nu_{i,j}= \nu_{i+\alpha m,j+\beta m}, \ \forall \, i,j,\alpha,\beta\in \mathbb{Z}  \} . 
	\end{eqnarray*}		
We also define the mean zero space 
	\[
\mathring{\mathcal V}_{\rm per}:=\left\{\nu\in {\mathcal V}_{\rm per} \ \middle| \overline{\nu} :=  \frac{h^2}{| \Omega|} \sum_{i,j=1}^m \nu_{i,j}  = 0\right\} .
	\]
The 4th order 2-D discrete Laplacian, $\Delta_{h,(4)} : {\mathcal V}_{\rm per}\rightarrow{\mathcal V}_{\rm per}$, is given by 
	\[
\Delta_{h,(4)} := {\cal D}_{x,(4)}^2 + D_{y, (4)}^2.
	\]
Now we define the following discrete inner products:  
	\begin{eqnarray*}
( \nu , \xi )_2 &:= h^2\sum_{i,j=1}^m \nu_{i,j}\psi_{i,j},\quad \nu,\, \xi\in {\mathcal V}_{\rm per} . 
	\end{eqnarray*}	
Suppose that $\zeta\in\mathring{\cal V}_{\rm per}$, then there is a unique solution $\msfT_h[\zeta]\in\mathring{\cal V}_{\rm per}$ such that $-\Delta_{h, (4)} \msfT_h[\zeta] = \zeta$. We often write, in this case, $\msfT_h[\zeta] = -\Delta^{-1}_{h,(4)} \zeta$. The discrete analog of the $\mathring{H}^{-1}_{\rm per}$ inner product is defined as 
	\[
(\zeta,\xi)_{-1} := \iprd{\zeta}{\msfT_h[\xi]}_2 = \iprd{\msfT_h[\zeta]}{\xi}_2,\quad \zeta,\, \xi\in\mathring{\cal V}_{\rm per}.
	\]
With the above machinery, if $\nu\in\mathring{\cal V}_{\rm per}$, then $\| \nu \|_{-1,h}^2 = (\nu , \nu )_{-1}$. If $\nu\in {\cal V}_{\rm per}$, then $\nrm{\nu}_2^2 := (\nu , \nu)_2$; $\nrm{\nu}_p^p := ( |\nu|^p , 1)_2$ ($1\le p< \infty$), and $\nrm{\nu}_\infty := \max_{1\le i\le m \atop 1\le j\le n}\left|\nu_{i,j}\right|$.

For $\phi, \psi \in \mathring{\cal V}_{\rm per}$, the following summation by parts formulas are available: 
	\begin{eqnarray} 
	  &&
- ( \phi , \Delta_{h, (4)} \psi )_2  = - ( \Delta_{h, (4)} \phi , \psi )_2 = ( \nabla_h \phi , \nabla_h \psi)_2 + \frac{h^2}{12} ( \Delta_h \phi , \Delta_h \psi )_2  , \label{summation-1} 
\\
  && 
  ( \msfT_h \phi , ( - \Delta_{h, (4)}) \psi )_2 = ( \phi , \psi )_2 .
	\label{summation-2}
	\end{eqnarray}
In turn, we denote the following norm for the discrete gradient, corresponding to the long stencil difference: 
\begin{eqnarray} 
  \| \nabla_{h,(4)} f \|_2^2 = \| \nabla_h f \|_2^2 + \frac{h^2}{12} \| \Delta_h f \|_2^2 . 
  \label{energy stability-2} 
\end{eqnarray}
In addition, the discrete $\nrm{ \, \cdot \, }_{H_h^1}$ and $\nrm{ \, \cdot \, }_{H_h^2}$ norms are given by  
	\begin{eqnarray}
\nrm{ f }_{H_h^1}^2 := \nrm{ f }_2^2 + \nrm{ \nabla_h f }_2^2 , \quad 
\nrm{ f }_{H_h^2}^2 := \nrm{ f }_{H_h^1}^2  + \nrm{ \Delta_h \phi }_2^2 .
	\label{discrete H1 H2 norm} 
%	\\
%\nrm{ f }_{H_h^4}^2 &:=& \nrm{ f }_{H_h^2}^2 + \nrm{ \nabla_h \Delta_h \phi }_2^2  + \nrm{ \Delta_h^2 \phi }_2^2  .  
%	\label{discrete H4 norm} 
	\end{eqnarray}
	
For any periodic grid function $\phi$, the discrete energy is introduced as  
\begin{eqnarray} 
    E_h (\phi) = \frac14 \| \phi \|_4^4 - \frac12 \| \phi \|_2^2 + \frac14 | \Omega | + \frac{\varepsilon^2}{2}  \| \nabla_{h,(4)} \phi \|_2^2 ,  \label{discrete energy-0} 
\end{eqnarray} 	

The following result, a discrete Sobolev embedding from $H_h^1$ to $\ell^6$, could be derived in the same manner as Lemma 2.1 in~\cite{cheng16a}, and Lemma 5.1 in~\cite{feng2017preconditioned}. The details are left to interested readers; also see the related analyses in~\cite{fengW17c}. 

\begin{lem} 
  For any periodic grid function $f$, we have 
\begin{eqnarray} 
  \| f \|_6 \le C ( \| f \|_2 + \| \nabla_h f \|_2 ) ,  \label{embedding-0} 
\end{eqnarray} 
for some constant $C$ only dependent on $\Omega$. 
\end{lem} 

The inequalities in the next lemma will play an important role in the energy stability and optimal rate convergence analysis. A direct calculation is not able to derive these inequalities; instead, a discrete Fourier analysis has to be applied in the derivation; the details will be left in Appendix~\ref{proof:Lemma 1}. 

\begin{lem}  \label{lem: inequality} 
  We have 
\begin{eqnarray} 
  &&
  \| f \|_2^2 \le \| f \|_{-1,h} \cdot \| \nabla_{h, (4)} f \|_2 ,   \quad \forall f \in  \mathring{\cal V}_{\rm per} ,  \label{inequality-0-1} 
\\
   &&
   \| \Delta_h f \|_2 \le \| \Delta_{h, (4)} f \|_2 , \quad \forall f \in {\cal V}_{\rm per} .  
   \label{inequality-0-2}    
\end{eqnarray} 
\end{lem}

\subsection{The fully discrete scheme and the main theoretical results} 

A modified second order BDF temporal discretization is applied to the Cahn-Hilliard equation, combined with long stencil fourth order difference approximation in space:  
 \begin{eqnarray} 
   \hspace{-0.3in} 
 \frac{\frac32 \phi^{k+1}- 2 \phi^{k}+ \frac12 \phi^{k-1}}{\dt} &=& \Delta_{h,(4)} \mu_h^{k+1} , 
 \label{scheme-BDF-CH-1}\\
   \hspace{-0.3in} 
 \mu_h^{k+1} &=& (\phi^{k+1})^3 - 2 \phi^k + \phi^{k-1} - \varepsilon^2 \Delta_{h,(4)}\phi^{k+1}  - A \dt \Delta_{h,(4)} ( \phi^{k+1} - \phi^k ) .  
 \label{scheme-BDF-CH-2}
 \end{eqnarray}
In comparison with the standard BDF algorithm, the concave diffusion term is updated explicitly, for the sake of unique solvability. In addition, a second order Douglas-Dupont-type regularization term is added in the chemical potential. Similar ideas could be found in~\cite{yan17}, with the finite element approximation in space, and~\cite{fengW17c}, in which the epitaxial thin film growth model is analyzed. 

We denote $\Phi$ as the exact solution for~\eqref{CH equation}, and the initial value is taken as $\phi^0_{i,j} = \Phi (x_i, y_j, t=0)$. In addition, it is noticed that~\eqref{scheme-BDF-CH-1}-\eqref{scheme-BDF-CH-2} is a two-step numerical method, so that a ``ghost" point extrapolation for $\phi^{-1}$ is needed. To preserve the second order accuracy in time, we apply the following approximation: 
\begin{eqnarray} 
  \phi^{-1} = \phi^0 - \dt \Delta_{h,(4)} \mu_h^0, \quad \mbox{with} \, \, \mu_h^0 := (\phi^0)^3 - \phi^0 - \varepsilon^2 \Delta_{h,(4)}\phi^0 .  \label{scheme-BDF-CH-initial-1}
\end{eqnarray}
A careful Taylor expansion indicates an $O (\dt^2 + h^4)$ accuracy for such an approximation: 
\begin{eqnarray} 
  \| \phi^{-1} - \Phi^{-1} \|_2 \le C (\dt^2 + h^4). \label{scheme-BDF-CH-initial-2}
\end{eqnarray}  

\begin{thm} \label{CH solvability} 
  Given $\phi^k, \phi^{k-1} \in {\mathcal V}_{\rm per}$, with $\overline{\phi^k} = \overline{\phi^{k-1}}$, there exists a unique solution $\phi^{k+1} \in {\mathcal V}_{\rm per}$ for the numerical scheme~\eqref{scheme-BDF-CH-1}-\eqref{scheme-BDF-CH-2}. And also, this scheme is mass conservative, i.e., $\overline{\phi^k} \equiv \overline{\phi^0} := \beta_0$, for any $k \ge 0$. 
\end{thm}  

	\begin{thm} \label{CH-energy stability}
For $k \ge 1$, we introduce
	\begin{equation}
\mathcal{E}_h (\phi^{k+1}, \phi^k) := E_h ( \phi^{k+1})+\frac{1}{4\dt} \| \phi^{k+1}- \phi^k\|_{-1,h}^2 + \frac{1}{2} \| \phi^{k+1} - \phi^k \|_2^2 . 
	\label{discrete energy}
	\end{equation}
For $A \ge \frac{1}{16}$, a modified energy-decay property is available for the numerical scheme~\eqref{scheme-BDF-CH-1}-\eqref{scheme-BDF-CH-2}: 
	\begin{equation}
\mathcal{E}_h ( \phi^{k+1}, \phi^k) 
%+\tau\left(1-\frac{1}{16 A} \right)\nrm{\frac{u_h^{n+1}-u_h^n}{\tau}}_{-1,h}^2 
\le \mathcal{E}_h ( \phi^k, \phi^{k-1}) .
	\label{CH-eng stab-est}
	\end{equation}
	\end{thm}
	
\begin{rem}
The energy stability for a gradient flow has always played an essential role in the accuracy of long time numerical simulation. Originated from the pioneering references~\cite{elliott93, eyre98}, the related works could also be found for many related physical models, such as  the phase field crystal (PFC) equation and the modified version~\cite{baskaran13a, baskaran13b, hu09, wang11a, wise09};  epitaxial thin film growth models~\cite{chen12, chen14, shen12, wang10a}; non-local gradient model~\cite{guan14b, guan14a}; the Cahn-Hilliard model coupled with fluid flow~\cite{chen16, diegel15a, feng12, LiuY17, wise10}; etc. Meanwhile, most of these works are associated with either the second order centered difference or finite element spatial approximations; this article is the first work to justify the energy stability for a fourth order finite difference scheme, combined with a second order temporal accuracy. 
\end{rem}
	
As a direct consequence of the energy stability, a uniform in time $H_h^1$ bound for the numerical solution is given as follows. 

	\begin{cor} \label{CH: H^1 bound} 
Suppose that the initial data are sufficiently regular so that
	\[
E_h (\phi^0) + \frac{\dt}{4} \| \nabla_{h, (4)} \mu_h^0 \|_2^2 + \frac{\dt^2}{2} \| \Delta_{h, (4)} \mu_h^0 \|_2^2 \le \tilde{C}_0,
	\]
for some $\tilde{C}_0$ that is independent of $h$, and $A\ge \frac{1}{16}$. Then we have the following uniform (in time) $H_h^1$ bound for the numerical solution: 
	\begin{equation}
\nrm{ \phi^m}_{H_h^1} \le \tilde{C}_1 ,  \quad \forall m \ge 1 , \label{CH-H1 stab-0}
	\end{equation} 
in which $\tilde{C}_1$ only depends on $\Omega$, $\varepsilon$ and $\tilde{C}_0$, independent on $h$, $\dt$ and final time.  
	\end{cor}
	
With an initial data with sufficient regularity, we could assume that the exact solution has regularity of class $\mathcal{R}$: 
		\begin{equation}
	 \Phi \in \mathcal{R} := H^3 (0,T; C^0) \cap H^2 (0,T; H^4) \cap L^\infty (0,T; H^8).
			\label{assumption:regularity.1}
		\end{equation}  
		
\begin{thm}  \label{thm:convergence}
    	Given initial data $\Phi_0 \in H_{\rm per}^8 (\Omega)$, suppose the exact solution for Cahn-Hilliard equation~\eqref{CH equation} is of regularity class $\mathcal{R}$. Then, provided $\dt$ and $h$ are sufficiently small, for all positive integers $n$, such that $n \dt \le T$, we have
    	\begin{equation}
    		\| \Phi^n - \phi^n \|_2 +  ( \varepsilon^2 \dt   \sum_{m=1}^{n} \| \Delta_h ( \Phi^m - \phi^m ) \|_2^2 )^{1/2}  \le C ( \dt^2 + h^4 ),   \label{CH-convergence-0}
    	\end{equation}
where $C>0$ is independent of $\dt$ and $h$.
\end{thm}

\section{The detailed proof} \label{sec:proof}

\subsection{Proof of Theorem~\ref{CH solvability}: unique solvability} 

\begin{proof} 
We see that the scheme \eqref{scheme-BDF-CH-1}-\eqref{scheme-BDF-CH-2} can be rewritten as 
\begin{eqnarray} 
\mathcal{N}_h [\phi] &=& f , \label{eqn:operator4th2nd}
\\
  \mbox{with} \, \, 
\mathcal{N}_h [\phi] &:=&-\Delta_{h,(4)}^{-1}\left( \frac32 \phi - 2 \phi^{k} + \frac12\phi^{k-1}\right) + \dt \phi^3 - \dt (A \dt+\varepsilon^2) \Delta_{h,(4)} \phi, \nonumber 
\\
 f &:=&  -2 \dt \phi^k + \dt\phi^{k-1} + A \dt^2 \Delta_{h,(4)}\phi^{k} . \nonumber 
\end{eqnarray}
Meanwhile, the nonlinear equation \eqref{eqn:operator4th2nd} can be recast as a minimization problem for the following discrete energy functional: 
\begin{equation}
F_h [\phi] :=  \frac{1}{3} \nrm{ \frac32\phi - 2 \phi^{k} + \frac12 \phi^{k-1}}_{-1,h}^2  + \frac{\dt}{4} \nrm{\phi}_{4}^4 + \frac{\dt}{2} (A \dt+ \varepsilon^2) \nrm{\nabla_{h,(4)} \phi}_{2}^2 - (f ,\phi)_2 , \label{eqn:min-energy-4th-2nd}
\end{equation}
for any $\phi \in \mathcal{V}_{\rm per}$. In turn, the strong convexity of $F_h$ (in terms of $\phi$), over the hyperplane of $\bar{\phi} = \beta_0$, implies a unique numerical solution for \eqref{scheme-BDF-CH-1}-\eqref{scheme-BDF-CH-2}. 

  By taking a discrete summation of \eqref{scheme-BDF-CH-1}, and making use of the fact that $\overline{\Delta_{h, (4)} \mu_h^{k+1} } =0$, as well as the mass conservation of the previous time steps: $\overline{\phi^k} = \overline{\phi^{k-1}} = \beta_0$, we are able to conclude that $\overline{\phi^{k+1}} = \beta_0$, for any $k \ge 0$. 
\end{proof} 

\subsection{Proof of Theorem~\ref{CH-energy stability}: energy stability} 

\begin{proof} 
  Since $\phi^{k+1} - \phi^k \in \mathring{\cal V}_{\rm per}$, we take a discrete inner product with (\ref{scheme-BDF-CH-1}) by $(-\Delta_{h, (4)})^{-1} (\phi^{k+1} - \phi^k)$, with the following inequalities derived: 
\begin{eqnarray} 
  && 
  \left(  \frac{\frac32 \phi^{k+1} - 2 \phi^k + \frac12 \phi^{k-1}}{2 \dt} , 
  (-\Delta_{h, (4)})^{-1} (\phi^{k+1} - \phi^k)  \right)_2   \nonumber 
\\
  &=&   
  \frac{1}{\dt}  \left(  
   \frac32 \nrm{ \phi^{k+1} - \phi^k }_{-1,h}^2 
  - \frac12 ( \phi^k - \phi^{k-1} , \phi^{k+1} - \phi^k )_{-1} \right) \nonumber 
\\
  &\ge&  
  \frac{1}{\dt}  \left(  
    \frac54 \nrm{ \phi^{k+1} - \phi^k }_{-1,h}^2 
  - \frac14 \nrm{ \phi^k - \phi^{k-1} }_{-1,h}^2  \right) , 
    \label{scheme-BDF-stability-1} 
\\
  &&
   \left(  - \Delta_{h, (4)} ( (\phi^{k+1})^3 )  , 
  (-\Delta_{h, (4)})^{-1} (\phi^{k+1} - \phi^k)  \right)_2  \nonumber 
\\
  &=&
  \left(  (\phi^{k+1})^3  ,  \phi^{k+1} - \phi^k  \right)_2
  \ge \frac14 (  \| \phi^{k+1} \|_4^4 - \| \phi^k \|_4^4 )  , 
    \label{scheme-BDF-stability-2} 
\\
  &&
  \left(  \Delta_{h, (4)}^2 \phi^{k+1}  , 
  (-\Delta_{h, (4)})^{-1} (\phi^{k+1} - \phi^k)  \right)_2  
  =  \left( -  \Delta_{h, (4)} \phi^{k+1}  , 
   \phi^{k+1} - \phi^k  \right)_2 \nonumber 
\\
  &=& 
   \frac12 \left(  \| \nabla_{h, (4)} \phi^{k+1} \|^2 
  - \| \nabla_{h, (4)} \phi^k \|^2   
  + \| \nabla_{h, (4)} ( \phi^{k+1} - \phi^k ) \|^2 \right)  , 
    \label{scheme-BDF-stability-3} 
\\
  &&
  \dt \left(  \Delta_{h, (4)}^2 ( \phi^{k+1} - \phi^k ) , 
  (-\Delta_{h, (4)})^{-1} (\phi^{k+1} - \phi^k)  \right)_2  
  =  \dt \| \nabla_{h, (4)}  ( \phi^{k+1} - \phi^k ) \|^2  , 
    \label{scheme-BDF-stability-4}  
\\
  &&
  \left(  \Delta_{h, (4)}  ( 2 \phi^k - \phi^{k-1})  , 
  (-\Delta_{h, (4)})^{-1} (\phi^{k+1} - \phi^k)  \right)_2  
  = - \left(  2 \phi^k - \phi^{k-1}  , 
  \phi^{k+1} - \phi^k)  \right)_2   \nonumber 
\\
  &\ge&   
    - \frac12 \left(  \| \phi^{k+1} \|_2^2 - \| \phi^k \|_2^2   \right)  
  - \frac12 \| \phi^k - \phi^{k-1} \|_2^2  .  
    \label{scheme-BDF-stability-5}    
\end{eqnarray} 
Meanwhile, an application of Cauchy inequality indicates the following estimate: 
\begin{eqnarray} 
  &&
   \frac{1}{\dt} \nrm{ \phi^{k+1} - \phi^k }_{-1,h}^2
   + A \dt \| \nabla_{h, (4)}( \phi^{k+1} - \phi^k ) \|_2^2   \nonumber 
\\
  &\ge& 
   2 A^{1/2} \nrm{ \phi^{n+1} - \phi^n }_{-1,h} 
   \cdot \| \nabla_{h, (4)}( \phi^{n+1} - \phi^n ) \|_2  
   \ge   2 A^{1/2}  \| \phi^{n+1} - \phi^n \|_2^2 , 
  \label{scheme-BDF-stability-6}    
\end{eqnarray} 
in which~\eqref{inequality-0-1} (in Lemma~\ref{lem: inequality}) has been used in the second step. In turn, a combination of (\ref{scheme-BDF-stability-1})-(\ref{scheme-BDF-stability-5}) and (\ref{scheme-BDF-stability-6}) yields 
\begin{eqnarray} 
  &&
  E_h (\phi^{k+1}) - E_h (\phi^k) 
  + \frac{1}{4 \dt}  \left(  \| \phi^{k+1} - \phi^k \|_{-1,h}^2 
  - \| \phi^k - \phi^{k-1} \|_{-1,h}^2  \right)  \nonumber 
\\
  &&   
  + \frac12 \left( \| \phi^{n+1} - \phi^n \|_2^2 
  - \| \phi^n - \phi^{n-1} \|_2^2  \right) 
  \le ( - 2 A^{1/2} + \frac12 ) \| \phi^{n+1} - \phi^n \|_2^2  \le 0 ,   \quad \mbox{if $A \ge \frac{1}{16}$} .   \label{scheme-BDF-stability-7}    
\end{eqnarray} 
Then we arrive at~\eqref{CH-eng stab-est}, provided that $A \ge \frac{1}{16}$. This completes the proof of Theorem~\ref{CH-energy stability}. 
\end{proof} 

\subsection{Proof of Corollary~\ref{CH: H^1 bound}: uniform in time $H_h^1$ bound} 

\begin{proof} 
As a result of \eqref{CH-eng stab-est}, the following energy bound is available: 
\begin{eqnarray} 
  E_h (\phi^m) &\le& \mathcal{E}_h (\phi^{m}, \phi^{m-1}) \le   \mathcal{E}_h (\phi^0, \phi^{-1})  = E_h (\phi^0) + \frac{1}{4 \dt}  \| \phi^0 - \phi^{-1} \|_{-1,h}^2    
  + \frac12 \| \phi^0 - \phi^{-1} \|_2^2   \nonumber 
\\
  &=& 
  E_h (\phi^0) + \frac{\dt}{4} \| \nabla_{h, (4)} \mu_h^0 \|_2^2 + \frac{\dt^2}{2} \| \Delta_{h, (4)} \mu_h^0 \|_2^2 \le \tilde{C}_0 , \quad \forall m \ge 1. \label{CH-H1 bound-1} 
\end{eqnarray}
On the other hand, the point-wise quadratic inequality, $\frac18 \phi^4 - \frac12 \phi^2 \ge - \frac12$, implies that 
\begin{eqnarray} 
  \frac18 \| \phi^m \|_4^4 - \frac12 \| \phi^m \|_2^2 \ge - \frac12 | \Omega | . 
  \label{CH-H1 bound-2} 
\end{eqnarray}
Its substitution into~\eqref{CH-H1 bound-1} yields 
\begin{eqnarray} 
  \frac12 \| \phi^m \|_2^2 +  \frac{\varepsilon^2}{2}  \| \nabla_{h,(4)} \phi^m \|_2^2 
  \le \tilde{C}_0 + \frac34 | \Omega | ,  \, \, \mbox{so that} \, \, 
  \| \phi^m \|_2^2 +   \| \nabla_{h,(4)} \phi^m \|_2^2 
  \le 2 \varepsilon^{-2} (\tilde{C}_0 + \frac34 | \Omega |) . 
   \label{CH-H1 bound-3} 
\end{eqnarray}
In turn, we arrive at 
\begin{eqnarray} 
  &&
  \| \phi^m \|_{H_h^1}^2  \le  \| \phi^m \|_2^2 +   \| \nabla_h \phi^m \|_2^2  
   \le  \| \phi^m \|_2^2 +   \| \nabla_{h,(4)} \phi^m \|_2^2 
  \le 2 \varepsilon^{-2} (\tilde{C}_0 + \frac34 | \Omega |) ,  \nonumber 
\\
  &&  \mbox{so that} \quad 
  \| \phi^m \|_{H_h^1}  
  \le \varepsilon^{-1} \Bigl( 2 (\tilde{C}_0 + \frac34 | \Omega |)  \Bigr)^{1/2} := \tilde{C}_1 ,  \quad \forall m \ge 1 , 
   \label{CH-H1 bound-4}    
\end{eqnarray}
in which the second step comes from an obvious fact that $\| \nabla_h \phi^m \|_2 \le \| \nabla_{h, (4)} \phi^m \|_2$. This completes the proof of Corollary~\ref{CH: H^1 bound}. 
\end{proof} 

\begin{rem}  \label{rem:L6 est}
As a combination of the uniform in time $H_h^1$ bound~\eqref{CH-H1 stab-0} and the discrete Sobolev embedding inequality~\eqref{embedding-0}, we arrive at a uniform in time $\ell^6$ estimate for the numerical solution: 
\begin{eqnarray} 
  \| \phi^m \|_6 \le C \tilde{C}_1 ,  \quad \forall m \ge 1 .  \label{CH-L6 est-0} 
\end{eqnarray} 
This estimate will be useful in the convergence analysis presented below.  
\end{rem}

\subsection{Proof of Theorem~\ref{thm:convergence}: the $\ell^2 (0,T; \ell^2) \cap \ell^2 (0, T; H_h^2)$ convergence analysis} 

Before the $\ell^2 (0,T; \ell^2) \cap \ell^2 (0, T; H_h^2)$ convergence analysis, we recall a modified version of discrete Gronwall inequality, excerpted from~\cite{diegel17}; this result will be used in the convergence estimate, due to the 2nd order BDF stencil. 

\begin{lem} \label{lem:discrete Gronwall} \cite{diegel17} 
Fix $T>0$. Let $M$ be a positive integer, with $\dt \le \frac{T}{M}$. Suppose $\left\{a_m\right\}_{m=0}^M$, $\left\{b_m\right\}_{m=0}^M$ and $\left\{c_m\right\}_{m=0}^{M-1}$ are non-negative sequences such that $\dt \sum_{m=0}^{M-1} c_m \le D_1$, with $D_1$  independent of $\dt$ and $M$.  Suppose that, for all $\dt>0$ and for some constant $0 < \alpha < 1$, 
	\begin{equation}
a_\ell + \dt \sum_{m=0}^{\ell} b_m \le D_2 +\dt \sum_{m=0}^{\ell-1} c_m \sum_{j=0}^m \alpha^{m-j} a_j ,  \quad \forall \, 1\le \ell \le  M , 
	\label{eq:gronwall-assumption}
	\end{equation}
where $D_2>0$ is a constant independent of $\dt$ and $M$.  Then, for all $\dt>0$, 
	\begin{equation}
a_\ell +\dt \sum_{m=0}^{\ell} b_m \le ( D_2 + a_0 D_1) \exp\left(\frac{D_1}{1-\alpha}\right) ,  \quad \forall \, 1\le \ell \le  M . 
	\label{eq:gronwall-conclusion}
	\end{equation}
\end{lem}

\begin{proof} 
For $\Phi \in \mathcal{R}$, a careful consistency analysis indicates the following truncation error estimate: 
\begin{eqnarray}
   \frac{\frac32 \Phi^{k+1} - 2 \Phi^k + \frac12 \Phi^{k-1}}{\dt}  &=& \Delta_{h,(4)} \Bigl(  (\Phi^{k+1})^3 - 2 \Phi^k + \Phi^{k-1} - \varepsilon^2 \Delta_{h,(4)} \Phi^{k+1}  \nonumber 
\\
  &&  \qquad 
  - A \dt \Delta_{h,(4)} ( \Phi^{k+1} - \Phi^k ) \Bigr) + \tau^{k+1} ,  
	\label{CH-consistency-1}
	\end{eqnarray} 
with $\| \tau^{k+1} \|_2 \le C (\dt^2 + h^4)$. The derivation of~\eqref{CH-consistency-1} is accomplished with the help of Proposition~\ref{thm-2d} and other related estimates; the details are left to interested readers. 

The numerical error function is defined at a point-wise level: 
\begin{eqnarray} 
  \tilde{\phi}^k := \Phi^k - \phi^k ,  \quad \forall m \ge 0 ,   \label{CH-error function-1}
\end{eqnarray} 
In turn, subtracting the numerical scheme \eqref{scheme-BDF-CH-1}-\eqref{scheme-BDF-CH-2} from (\ref{CH-consistency-1}) gives 
\begin{eqnarray}
   \frac{\frac32 \tilde{\phi}^{k+1} - 2 \tilde{\phi}^k + \frac12 \tilde{\phi}^{k-1}}{\dt}  &=& 
   \Delta_{h,(4)} \Bigl(  {\cal NL} (\Phi^{k+1}, \phi^{k+1}) - 2 \tilde{\phi}^k + \tilde{\phi}^{k-1} - \varepsilon^2 \Delta_{h,(4)} \tilde{\phi}^{k+1}  \nonumber 
\\
  &&  \quad  
  - A \dt \Delta_{h,(4)} ( \tilde{\phi}^{k+1} - \tilde{\phi}^k ) \Bigr) + \tau^{k+1} ,  
	\label{CH-consistency-2-1} 
\\
   \mbox{with} \quad   
  {\cal NL} (\Phi^{k+1}, \phi^{k+1})  &=& ( (\Phi^{k+1})^2 + \Phi^{k+1} \phi^{k+1} 
  + (\phi^{k+1})^2 ) \tilde{\phi}^{k+1} .  \label{CH-consistency-2-2}     
	\end{eqnarray} 
	
Taking a discrete inner product with \eqref{CH-consistency-2-1}-\eqref{CH-consistency-2-2} by $ \tilde{\phi}^{k+1}$, with a repeated application of summation by parts, we get 
\begin{eqnarray} 
  \hspace{-0.2in}
  &&
   \Bigl( \frac32 \tilde{\phi}^{k+1} - 2 \tilde{\phi}^k + \frac12 \tilde{\phi}^{k-1} , \tilde{\phi}^{k+1} \Bigr)_2  + \varepsilon^2 \dt \| \Delta_{h,(4)} \tilde{\phi}^{k+1}  \|_2^2 
   +   A \dt^2 \Bigl( \Delta_{h,(4)} ( \tilde{\phi}^{k+1} - \tilde{\phi}^k ) , 
   \Delta_{h,(4)} \tilde{\phi}^{k+1}  \Bigr)_2  \nonumber 
\\
  \hspace{-0.2in}  
  &=&
  \dt \left( {\cal NL} (\Phi^{k+1}, \phi^{k+1}) ,  \Delta_{h,(4)} \tilde{\phi}^{k+1}  \right)_2 
  + \dt ( \tilde{\phi}^{k+1} , \tau^{k+1} )_2 .   \label{CH-convergence-1} 
\end{eqnarray} 
The time marching term could be analyzed as follows:  
\begin{eqnarray} 
  \hspace{-0.3in}
  &&
   \left( \frac32 \tilde{\phi}^{k+1} - 2 \tilde{\phi}^k + \frac12 \tilde{\phi}^{k-1} , 
    \tilde{\phi}^{k+1} \right)_2 
   = \frac32 \left( \tilde{\phi}^{k+1} - \tilde{\phi}^k , \tilde{\phi}^{k+1} \right)_2
   - \frac12 \left( \tilde{\phi}^k - \tilde{\phi}^{k-1} , \tilde{\phi}^{k+1} \right)_2       
    \nonumber 
\\
  \hspace{-0.3in}  
  &\ge&  
  \left( \frac34 \| \tilde{\phi}^{k+1} \|_2^2 - \frac14 \| \tilde{\phi}^k \|_2^2  \right) 
 -  \left( \frac34 \| \tilde{\phi}^k \|_2^2 - \frac14 \| \tilde{\phi}^{k-1} \|_2^2  \right)   
    + \frac12 \left( \| \tilde{\phi}^{k+1} - \tilde{\phi}^k \|_2^2  
   -  \|  \tilde{\phi}^k - \tilde{\phi}^{k-1}  \|_2^2  \right)  . 
   \label{CH-convergence-2}
\end{eqnarray} 
The third term on the left hand side of~\eqref{CH-convergence-1} could be handled as follows:
\begin{eqnarray} 
  \Bigl( \Delta_{h,(4)} ( \tilde{\phi}^{k+1} - \tilde{\phi}^k ) , 
   \Delta_{h,(4)} \tilde{\phi}^{k+1}  \Bigr)_2 
   \ge \frac12 ( \| \Delta_{h,(4)} \tilde{\phi}^{k+1} \|_2^2 
   - \| \Delta_{h,(4)} \tilde{\phi}^k \|_2^2 ) .  \label{CH-convergence-3}
\end{eqnarray} 
The term associated with the local truncation error could be bounded with the help of Cauchy inequality: 
\begin{eqnarray} 
  ( \tilde{\phi}^{k+1} , \tau^{k+1} )_2  \le  \| \tilde{\phi}^{k+1}\|_2 \cdot \|  \tau^{k+1} \|_2  
  \le  \frac12 ( \| \tilde{\phi}^{k+1}\|_2^2 + \|  \tau^{k+1} \|_2^2 ) .  \label{CH-convergence-4}
\end{eqnarray} 

For the nonlinear error term, we begin with an application of discrete H\"older inequality: 
\begin{eqnarray} 
  \| {\cal NL} (\Phi^{k+1} \|_2 &\le& \| (\Phi^{k+1})^2 + \Phi^{k+1} \phi^{k+1} 
  + (\phi^{k+1})^2 \|_3 \cdot \| \tilde{\phi}^{k+1} \|_6  \nonumber 
\\
  &\le&
  C ( \| \Phi^{k+1} \|_6^2 + \| \phi^{k+1} \|_6^2 )  \| \tilde{\phi}^{k+1} \|_6  
  \le C ( (C^*)^2 + \tilde{C}_1^2 )   \| \tilde{\phi}^{k+1} \|_6 ,  \label{CH-convergence-5-1}
\end{eqnarray} 
in which the estimates $\| \Phi^{k+1} \|_6 \le C^*$, and $\| \phi^{k+1} \|_6 \le C \tilde{C}_1$ (which comes from the uniform in time estimate~\eqref{CH-L6 est-0} for the numerical solution, as given by Remark~\ref{rem:L6 est}), have been used. On the other hand, we make use of the discrete Sobolev embedding inequality~\eqref{embedding-0} again, and obtain
\begin{eqnarray} 
  \| \tilde{\phi}^{k+1} \|_6  &\le& C ( \| \tilde{\phi}^{k+1} \|_2 + \| \nabla_h \tilde{\phi}^{k+1} \|_2 ) 
  \le C ( \| \tilde{\phi}^{k+1} \|_2 + \| \nabla_{h, (4)} \tilde{\phi}^{k+1} \|_2 )   \nonumber 
\\
  &\le&
  C ( \| \tilde{\phi}^{k+1} \|_2 + \| \tilde{\phi}^{k+1} \|_2^{1/2} \cdot   
  \| \Delta_{h, (4)} \tilde{\phi}^{k+1} \|_2^{1/2} )  ,  \label{CH-convergence-5-2}
\end{eqnarray} 
in which the second step comes from the fact that $ \| \nabla_h \tilde{\phi}^{k+1} \|_2  \le \| \nabla_{h, (4)} \tilde{\phi}^{k+1} \|_2$, while the last step is based on the following estimate: 
\begin{eqnarray} 
  \| \nabla_{h, (4)} \tilde{\phi}^{k+1} \|_2^2 = ( \tilde{\phi}^{k+1} , - \Delta_{h, (4)} \tilde{\phi}^{k+1} )_2  \le  
     \| \tilde{\phi}^{k+1} \|_2  \cdot  \| \Delta_{h, (4)} \tilde{\phi}^{k+1} \|_2 )  . 
     \label{CH-convergence-5-3}
\end{eqnarray} 
Consequently, a substitution of~\eqref{CH-convergence-5-2} into \eqref{CH-convergence-5-1} yields 
\begin{eqnarray} 
  \| {\cal NL} (\Phi^{k+1} \|_2    \le  \tilde{C}_2  
  ( \| \tilde{\phi}^{k+1} \|_2 + \| \tilde{\phi}^{k+1} \|_2^{1/2} \cdot   
  \| \Delta_{h, (4)} \tilde{\phi}^{k+1} \|_2^{1/2} )  ,  \label{CH-convergence-5-4}
\end{eqnarray}   
with $\tilde{C}_2 = C ( (C^*)^2 + \tilde{C}_1^2 )$. In turn, a bound for the nonlinear error inner product term could be derived: 
\begin{eqnarray} 
  \left( {\cal NL} (\Phi^{k+1}, \phi^{k+1}) ,  \Delta_{h,(4)} \tilde{\phi}^{k+1}  \right)_2  
  &\le& \| {\cal NL} (\Phi^{k+1} \|_2  \cdot \| \Delta_{h,(4)} \tilde{\phi}^{k+1} \|_2  \nonumber 
\\
  &\le&  \tilde{C}_2  
  ( \| \tilde{\phi}^{k+1} \|_2  \cdot \| \Delta_{h,(4)} \tilde{\phi}^{k+1} \|_2 
  + \| \tilde{\phi}^{k+1} \|_2^{1/2} \cdot   
  \| \Delta_{h, (4)} \tilde{\phi}^{k+1} \|_2^{3/2} )  \nonumber 
\\
  &\le& 
  \tilde{C}_{3, \varepsilon} \| \tilde{\phi}^{k+1} \|_2^2   
  + \frac{\varepsilon^2}{2} \| \Delta_{h, (4)} \tilde{\phi}^{k+1} \|_2^2  ,  
  \label{CH-convergence-5-5}
\end{eqnarray} 
with the Young's inequality applied in the last step. 

Subsequently, a substitution of~\eqref{CH-convergence-2}-\eqref{CH-convergence-4} and \eqref{CH-convergence-5-5} into \eqref{CH-convergence-1} yields 
\begin{eqnarray} 
  && 
  {\cal G}^{k+1} - {\cal G}^k 
  + \frac{\varepsilon^2}{2} \dt \| \Delta_{h, (4)} \tilde{\phi}^{k+1} \|_2^2  
  \le ( \tilde{C}_{3, \varepsilon} + \frac12 ) \dt \| \tilde{\phi}^{k+1} \|_2^2  
  + \dt \| \tau^{k+1} \|_2^2 ,  \label{CH-convergence-6-1} 
\\
  && \mbox{with} \quad 
  {\cal G}^{k+1} := \frac34 \| \tilde{\phi}^{k+1} \|_2^2 - \frac14 \| \tilde{\phi}^k \|_2^2 
    + \frac12 \| \tilde{\phi}^{k+1} - \tilde{\phi}^k \|_2^2  
    + \frac{A}{2} \dt^2 \| \Delta_{h,(4)} \tilde{\phi}^{k+1} \|_2^2  . 
    \label{CH-convergence-6-2} 
\end{eqnarray} 
Meanwhile, for the term $\| \tilde{\phi}^{k+1} \|_2^2$, the following inductive estimate is available: 
\begin{eqnarray} 
   \| \tilde{\phi}^{k+1} \|_2^2 &\le& \frac43 {\cal G}^{k+1} + \frac13 \| \tilde{\phi}^k \|_2^2  
   \le \frac43 {\cal G}^{k+1} + \frac49 {\cal G}^k + \frac19 \| \tilde{\phi}^{k-1} \|_2^2  \le ... 
   \nonumber
\\
  &\le& 
  \frac43 \sum_{i=0}^{k+1} (\frac13)^i {\cal G}^{k+1-i} .  \label{CH-convergence-6-3} 
\end{eqnarray} 
This in turn leads to the following inequality
\begin{eqnarray}  
  {\cal G}^{k+1} - {\cal G}^k 
  + \frac{\varepsilon^2}{2} \dt \| \Delta_{h, (4)} \tilde{\phi}^{k+1} \|_2^2 
    \le \frac43 ( \tilde{C}_{3, \varepsilon} + \frac12 ) \dt 
    \sum_{i=0}^{k+1} (\frac13)^i {\cal G}^{k+1-i}    
 + \dt \| \tau^{k+1} \|_2^2 .  \label{CH-convergence-6-4} 
\end{eqnarray}  
Therefore, with an application of Lemma~\ref{lem:discrete Gronwall}, and making use of the fact that $\| \tau^{k+1} \|_2 \le C (\dt^2 + h^4)$, we conclude that 
\begin{eqnarray} 
   {\cal G}^{k+1}  + \varepsilon^2 \dt \sum_{i=1}^{k+1} \| \Delta_{h, (4)} \tilde{\phi}^i \|_2^2  \le \hat{C} ( \dt^4 + h^8) ,   \label{CH-convergence-6-5} 
\end{eqnarray}  
with $\hat{C}$ independent on $\dt$ and $h$. Furthermore, with an application of~\eqref{CH-convergence-6-3}, we arrive at the desired convergence estimate: 
\begin{eqnarray} 
   \| \tilde{\phi}^{k+1} \|_2  + \Bigl( \varepsilon^2 \dt \sum_{i=1}^{k+1} \| \Delta_h \tilde{\phi}^i \|_2^2  \Bigr)^{1/2} \le C \hat{C}^{1/2} ( \dt^2 + h^4) ,   \label{CH-convergence-6-6} 
\end{eqnarray}  
This completes the proof of Theorem~\ref{thm:convergence}. 
\end{proof}

\section{Numerical results} \label{sec:numerical results}

The numerical implementation of the proposed fourth order finite difference scheme~\eqref{scheme-BDF-CH-1}-\eqref{scheme-BDF-CH-2} requires a nonlinear solver. Meanwhile, since the nonlinear term corresponds to a convex functional, as indicated by~\eqref{eqn:min-energy-4th-2nd}, some ideas associated with convex optimization could be efficiently applied. We use the precondition steepest descent (PSD) iteration to implement this numerical algorithm; such an iteration has been proposed and analyzed for the regularized p-Laplacian problem in recent works~\cite{feng2017preconditioned, fengW17c}, and its extension to Cahn-Hilliard-type problem is straightforward.

\subsection{Precondition steepest descent (PSD) solver} 
The main idea of the PSD solver is to use a linearized version of the nonlinear operator as a pre-conditioner, or in other words, as a metric for choosing the search direction.  A linearized version of the nonlinear operator $\mathcal{L}_h: \mathring{\mathcal V}_{\rm per} \to \mathring{\mathcal V}_{\rm per}$, is defined as follows: 
	\begin{eqnarray}
{\mathcal L}_h [\psi] := -\Delta_{h,(4)}^{-1} \psi + \dt \psi - \dt ( \varepsilon^2 + A  \dt ) \Delta_{h,(4)}^2 \psi.
	\end{eqnarray}
Given the current iterate $\phi^{(n)}\in {\mathcal V}_{\rm per}$, we define the following \emph{search direction} problem: find $d^{(n)} \in \mathring{\mathcal V}_{\rm per}$ such that
\[
{\mathcal L}_h [d^{(n)}]= f - \mathcal{N}_h [\phi^{(n)}]:=r^{(n)},
\]
where $r^{(n)}$ is the nonlinear residual of the $n^{\rm th}$ iterate $\phi^{(n)}$. This equation can be solved efficiently using the Fast Fourier Transform (FFT).

In turn, the next iterate is given by
	\begin{equation}
\phi^{(n+1)} = \phi^{(n)} + \overline{\alpha} d^{(n)},
	\end{equation}
where $\overline{\alpha}\in\mathbb{R}$ is the unique solution to the steepest descent line minimization problem 
	\begin{equation}
\overline{\alpha} := \operatorname*{argmax}_{\alpha\in\mathbb{R}} F_h [\phi^{(n)} + \alpha d^{(n)}]= \operatorname*{argzero}_{\alpha\in\mathbb{R}}\delta F_h [\phi^{(n)} + \alpha d^{(n)}] (d^{(n)}) .
	\label{eqn-search}
	\end{equation}

The geometric convergence analysis of the PSD solver has been established in~\cite{feng2017preconditioned}, for egularized p-Laplacian problem. For the Cahn-Hilliard-type problem \eqref{scheme-BDF-CH-1}-\eqref{scheme-BDF-CH-2}, the corresponding nonlinearity is weaker than that of the p-Laplacian  problem, and the geometric convergence rate: $\phi^{(n)} \to \phi^{k+1}$, (where $\phi^{k+1}$ is the exact numerical solution to \eqref{scheme-BDF-CH-1}-\eqref{scheme-BDF-CH-2}), as $n\to \infty$, is expected to be derived in a similar way; the details are left to interested readers.

\subsection{Convergence test for the numerical scheme}

In this subsection we perform some numerical experiments to support the theoretical results, using a uniform Cartesian grid and set periodic boundary conditions. In particular, it is observed that the search direction and Poisson-like equations can also be efficiently solved efficiently by using the Fourier pseudo-spectral method (see the related discussions in~\cite{Boyd2001, cheng2015fourier, GO1977, HGG2007}) and Fast Fourier Transform (FFT). 

To test the convergence rate, we choose the data such that the exact solution of \eqref{CH equation} on the square domain $\Omega=(0, 3.2) \times (0, 3.2)$:
\begin{eqnarray}\label{eqn:init1}
\phi_e(x,y,t) = \frac{1}{2\pi}\sin\big({2\pi x}/{3.2}\big)\cos\big({2\pi y}/{3.2}\big)\cos(t).
\end{eqnarray}
%We take a quartic refinement path for scheme \eqref{scheme-CS-CH1}-\eqref{scheme-CS-CH2} , i.e.  $s=Ch^4$. Then at the final time $T=0.32$, we expect the global error to be $\mathcal{O}(h^4)$ under the $\|\cdot\|_\infty$ and $\|\cdot\|_2$ norm, as $h, s\to 0$. The other parameter are given by $L_x=L_y=3.2$, $\epsilon=0.1$, $s=0.5h^4$. Similarly, 
We take a quadratic refinement path for scheme \eqref{scheme-BDF-CH-1}-\eqref{scheme-BDF-CH-2} , i.e.  $\dt=Ch^2$. The final time is taken as $T=0.32$, and we expect the global error to be $\mathcal{O}(h^4)$ under the $\|\cdot\|_\infty$ and $\|\cdot\|_2$ norm, as $h, \dt \to 0$. The other parameter are given by $L_x=L_y=3.2$, $\varepsilon=0.1$.

The error norms and the convergence rates can be found in Table~\ref{tab:cov2nd}, which confirms our theoretical results.

\begin{table}[!htb]
\begin{center}
\caption{The error norms and the convergence rates for the computed solution with scheme \eqref{scheme-BDF-CH-1}-\eqref{scheme-BDF-CH-2} .} \label{tab:cov2nd}
\begin{tabular}{ccccc}
%\hline &&\multicolumn{2}{c}{FAS}\\
%\hline &&\multicolumn{4}{c}{$$}&\multicolumn{4}{c}{MG}\\
\hline $h$&$\|\phi-\mathcal{I}_h\phi_e\|_\infty$ & Rate& $\|\phi-\mathcal{I}_h\phi_e\|_2$& Rate\\
\hline $\frac{3.2}{16}$& $ 5.4719\times 10^{-5}$&- &
$9.1073\times 10^{-5}$&- 
\\$\frac{3.2}{32}$& $3.5430\times 10^{-6}$ &3.95& 
  $5.7279\times 10^{-6}$ &3.99 
\\ $\frac{3.2}{64}$ & $2.2333\times 10^{-7}$&3.99 &
   $3.5846\times 10^{-7}$&4.00
\\ $\frac{3.2}{128}$&$1.3961\times 10^{-8}$&4.00&
   $2.2367\times 10^{-8}$&4.00
\\
\hline
\end{tabular}
\end{center}
\end{table}

\subsection{Numerical simulation of spinodal decomposition and energy dissipation}

We simulate the spinodal decomposition of a mixed binary fluid and present the energy dissipation in this subsection. The parameters are given by $L_x=L_y=12.8$, $\varepsilon=0.03$, $h=12.8/512$,. The initial data for this simulation is taken as a random field values $\phi^0_{i,j}=\bar{\phi}+0.1\cdot(2r_{i,j}-1)$, with an average composition $\bar{\phi}=0$ and $r_{i,j}\in[0,1]$. For the temporal step size $\dt$, we use increasing values of $\dt$ in the time evolution: $\dt = 0.01$ on the time interval $[0,2000]$ and $\dt = 0.04$ on the time interval $[2000,6000]$. Whenever a new time step size is applied, we initiate the two-step numerical scheme by  taking $\phi^{-1} = \phi^0$, with the initial data $\phi^0$ given by the final time output of the last time period. The snapshots of spinodal decomposition for the proposed numerical scheme, with second order accuracy in time and fourth order accuracy in space, can be found in Figure~\ref{fig:spinodal-2nd}.
 The corresponding energy decay plot is displayed in Figure~\ref{fig:energy-ch-2nd}. 
 %The log-log plots of energy evolution and the corresponding linear regression in Figure. \ref{fig:energy-ch-1st} shows that the energy indeed decays like $t^{-1/3}$ for the first-order-in-time scheme, which verifies the one-third power law. More precisely, the linear fit has the form $a_e t^{-b_e}$ with $a_e=2.3383, b_e=0.3378$. 
 
 The log-log plots of energy evolution and the corresponding linear regression in Figure~\ref{fig:energy-ch-2nd} shows that the energy indeed decays like $t^{-1/3}$ for the proposed scheme, which verifies the one-third power law. The detailed scaling ``exponent" is obtained using least squares fits of the computed data up to time $t=400$.  A clear observation of the $a_e t^{-b_e}$ scaling law can be made, %with different coefficients dependent upon $\varepsilon$, or, equivalently, the domain size, $L$. 
with $a_e = 2.3670$, $b_e=0.3332$. In other words, an almost perfect $t^{-1/3}$ energy dissipation law is confirmed by our numerical simulation.  
  
 The linear regressions is only taken up to $t = 6000$, since the saturation time would be of the order of $\varepsilon^{-2}$. 
%under the scaling that we have adopted for the model \cite{shen2012}.
%The evolutions of the mass in Figure. \ref{fig:energy-ch-1st} and Figure. \ref{fig:energy-ch-2nd} clearly confirm the mass conservative property.

	\begin{figure}
 	\begin{center}
\includegraphics[width=\textwidth]{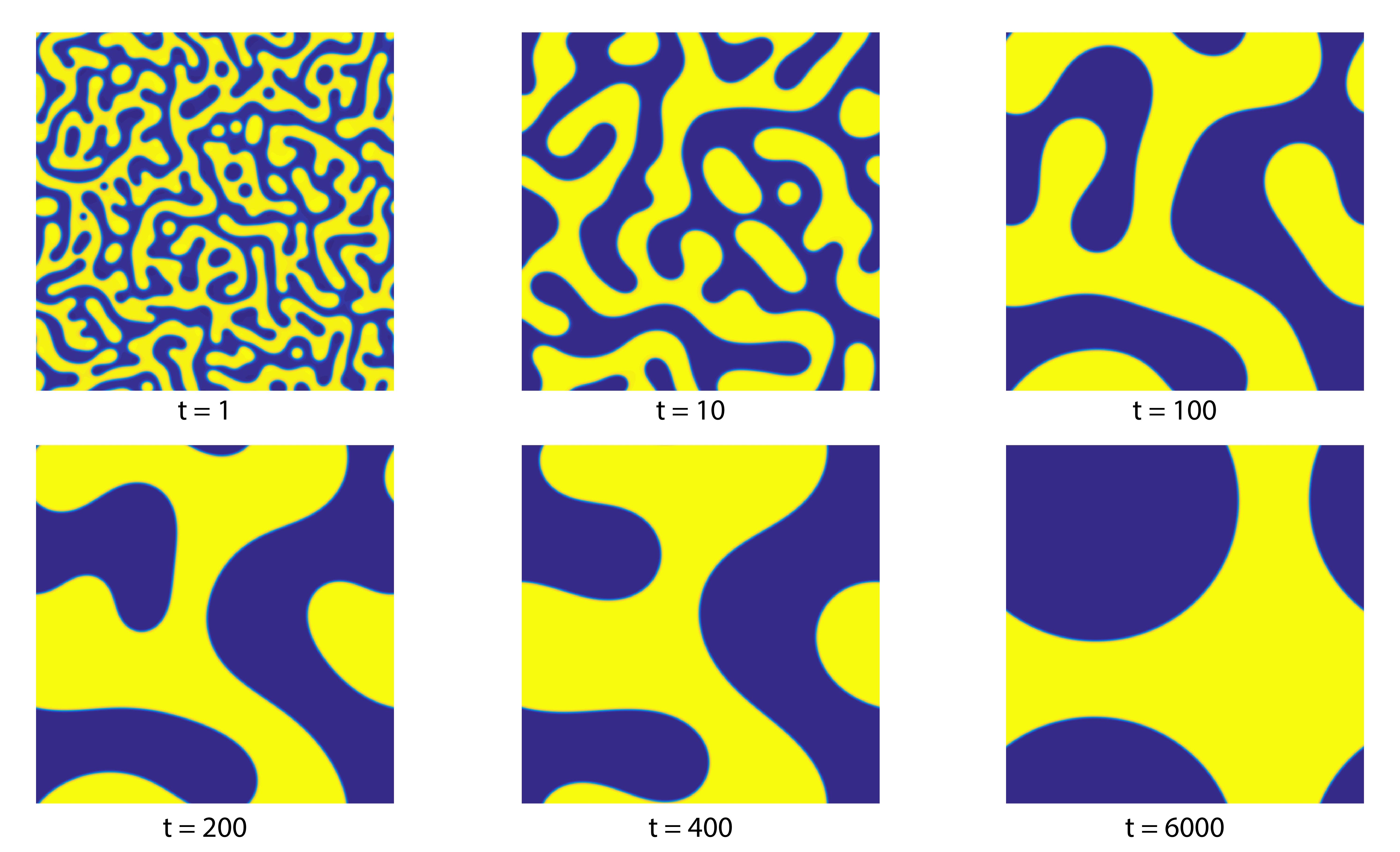}  
 	\end{center}
\caption{(Color online). Snapshots of spinodal decomposition of Cahn-Hilliard equation for a mixed binary fluid with scheme \eqref{scheme-BDF-CH-1}-\eqref{scheme-BDF-CH-2} in $\Omega=(0,12.8) \times (0, 12.8)$ at a sequence of time instants: 1, 10, 100, 200, 400 and 2000. The surface diffusion parameter is taken to be $\varepsilon=0.03$ and the time step size is $\dt=0.01$. }\label{fig:spinodal-2nd}	
	\end{figure}

\begin{figure}[!htp]
\centering
\includegraphics[width=3.5in, height=3.0in]{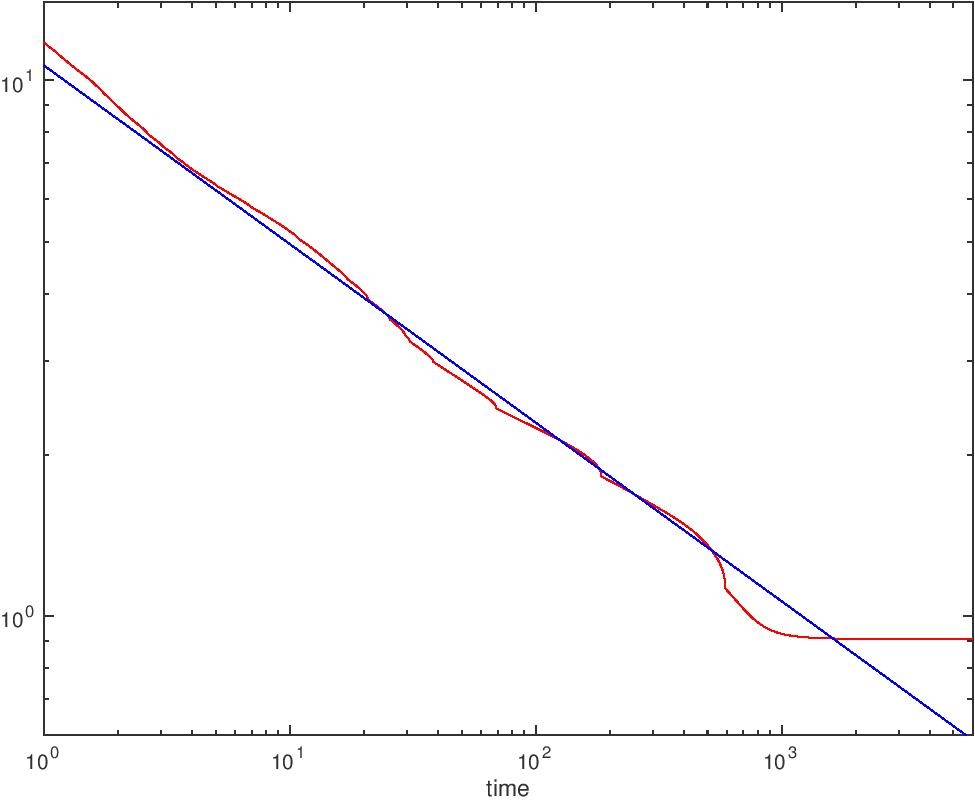}
\caption{The evolutions of discrete energy. The parameters are given in the text and in the caption of Figure \ref{fig:spinodal-2nd}. The energy decays like $t^{-1/3}$, which verifies the one-third power law. More precisely, the linear fit has the form $a_e t^{-b_e}$ with $a_e=2.3670, b_e=-0.3332$.}
\label{fig:energy-ch-2nd}
\end{figure}

\section{Concluding remarks} \label{sec:conclusion}

In this article, we propose and analyze an energy stable fourth order finite difference scheme for the Cahn-Hilliard equation, with second order temporal accuracy. As a preliminary truncation error estimate for the long stencil difference operator, over a uniform numerical grid with a periodic boundary condition, the discrete $\ell^2$ estimate only requires an $H^m$ regularity for the test function, which in turn results in a reduced regularity requirement. In the temporal approximation, we apply a modified BDF algorithm, combined with a second order extrapolation formula applied to the concave term. And also, a second order artificial Douglas-Dupont regularization is included in the numerical scheme, to ensure the energy stability at a discrete level. With such a careful construction, the unique solvability and energy stability are proved for the proposed numerical scheme, and a uniform in time $H_h^1$ bound for the numerical solution is established. As a result of this $H_h^1$ bound, we are able to derive an optimal rate convergence analysis for the proposed numerical scheme, in the $\ell^\infty (0,T; \ell^2) \cap \ell^2 (0,T; H_h^2)$ norm. In addition, a few numerical experiments have confirmed these theoretical results, and the numerical simulations results of spinodal decomposition in a mixed binary fluid have indicated an energy dissipation law of $t^{-1/3}$.

	\appendix

	\section{Proof of Lemma \ref{lem-1}}
	\label{proof:Lemma 1}
	
By substituting (\ref{expansion-2d-f-dis-2}), we see that the first 
inequality of (\ref{est-2d-coefficient-2}) is equivalent to 
\begin{equation} 
  \left| - 4 \mbox{sin}^2 (k \pi h/L) - \frac43 \mbox{sin}^4 (k \pi h/L) 
  + \frac{4 k^2\pi^2 h^2}{L^2} \right|
  \le C_1 h^6 \left( \frac{2 k \pi}{L} \right)^6 ,  
  \quad \forall \, \, 0 \le k \le N . 
  \label{est-2d-coefficient-4}
\end{equation}

  We denote $c_k = k \pi h/L$. Due to the fact that $h = \frac{L}{2 N+1}$, 
we have  
\begin{equation} 
  0 \le c_k \le \frac{\pi}{2} ,  \quad \forall \, \, 1 \le k \le N . 
  \label{est-2d-coefficient-4-2}
\end{equation}

  In turn, we need both the lower and upper bounds of $- 4 \mbox{sin}^2 c_k - \frac43 \mbox{sin}^4 c_k + 4 c_k^2$ to establish the estimate 
(\ref{est-2d-coefficient-4}). For the lower bound, the following inequality is observed:
\begin{equation} 
    h_1 (t) := \mbox{sin}^2 t + \frac13 \mbox{sin}^4 t  
  \le h_2 (t) := t^2 ,  \quad \forall \, \, t \ge 0 . 
  \label{est-2d-coefficient-lemma1-1}
\end{equation} 
The derivation of this inequality is based on the fact that 
\begin{equation} 
  h_1 (0) = h_2 (0) = 0 ,   \quad  h'_1 (0) = h'_2 (0) = 0  , 
\end{equation} 
and a careful comparison between their second order derivatives: 
\begin{equation} 
    h''_1 (t) = 2 \left( \mbox{cos}^2 t + \mbox{sin}^2 t \mbox{cos}^2 t 
  - \frac53 \mbox{sin}^4 t \right)  \le  h''_2 (t) = 2  ,   \quad \forall \, \, t \ge 0 . 
\end{equation} 
As a result of (\ref{est-2d-coefficient-lemma1-1}), we obtain the following 
lower bound: 
\begin{equation} 
    0 \le - 4 \mbox{sin}^2 c_k - \frac43 \mbox{sin}^4 c_k + 4 c_k^2 , 
   \quad \forall \, \, 0 \le k \le N .   
   \label{est-2d-coefficient-lemma1-2}
\end{equation} 

   For the upper bound, we begin with a Taylor expansion for $\mbox{sin} t$: 
\begin{equation} 
  \mbox{sin} t  \ge t - \frac{t^3}{3 !} = t - \frac{t^3}{6} \ge 0 ,  \quad 
  \forall \, \, 0 \le t \le \frac{\pi}{2} ,  \label{est-2d-coefficient-lemma1-3}
\end{equation}
in which the second inequality comes from the range that 
$0 \le t \le \frac{\pi}{2}$. Subsequently, we set $c_k = k \pi h/L$ and arrive at 
\begin{eqnarray} 
   4 \mbox{sin}^2 c_k + \frac43 \mbox{sin}^4 c_k 
  \ge 4  \left(  c_k - \frac{c_k^3}{6} \right)^2 
  +  \frac43  \left(  c_k - \frac{c_k^3}{6} \right)^4 
  = 4 c_k^2 - \frac{7}{9} c_k^6 + \frac{2}{9} c_k^8 
   - \frac{2}{81} c_k^{10}  + \frac{1}{972} c_k^{12}  ,   
  \label{est-2d-coefficient-lemma1-4}
\end{eqnarray}
for any $0 \le k \le N$.  In turn, the following estimate is available: 
\begin{eqnarray} 
    - 4 \mbox{sin}^2 c_k - \frac43 \mbox{sin}^4 c_k + 4 c_k^2  
  \le   \frac{7}{9} c_k^6 - \frac{2}{9} c_k^8 
   + \frac{2}{81} c_k^{10}  - \frac{1}{972} c_k^{12}  ,   
   \quad \forall \, \, 0 \le k \le N . 
  \label{est-2d-coefficient-lemma1-5}
\end{eqnarray}

 Consequently, a combination of (\ref{est-2d-coefficient-lemma1-2}) and 
(\ref{est-2d-coefficient-lemma1-5}) results in 
\begin{eqnarray} 
    \left| - 4 \mbox{sin}^2 c_k - \frac43 \mbox{sin}^4 c_k + 4 c_k^2  \right|
  \le   \frac{7}{9} c_k^6 - \frac{2}{9} c_k^8 
   + \frac{2}{81} c_k^{10}  - \frac{1}{972} c_k^{12}  ,   
   \quad \forall \, \, 0 \le k \le N . 
  \label{est-2d-coefficient-lemma1-6}
\end{eqnarray}
On the other hand, by the definition $c_k = k \pi h/L$, the 
following estimates can be derived: 
\begin{eqnarray} 
  &&
    c_k^6 = h^6 \left( k \pi /L \right)^6  
  \le \frac{1}{64}  h^6 \left( 2 k \pi /L \right)^6 ,  
  \quad \forall \, \, 0 \le k \le N , 
  \label{est-2d-coefficient-lemma1-7-1} 
\\
  &&
  c_k^m = c_k^6 \cdot c_k^{m-6} 
  \le \left( \frac{\pi}{2} \right)^{m-6} 
  \cdot \frac{1}{64} h^6 \left( 2 k \pi /L \right)^6  ,  
  \quad \forall \, \, m \ge 6 ,  \, \, 0 \le k \le N . 
  \label{est-2d-coefficient-lemma1-7-2} 
\end{eqnarray}
Finally, a combination of (\ref{est-2d-coefficient-lemma1-6}), 
(\ref{est-2d-coefficient-lemma1-7-1}) and (\ref{est-2d-coefficient-lemma1-7-2}) implies (\ref{est-2d-coefficient-4}), with an appropriate choice of $C_1$. The proof of the first part of Lemma \ref{lem-1} is finished. The second inequality can be derived in the same manner.

\section{Proof of Lemma \ref{lem-2} } \label{proof:Lemma 2}

We see that the expansion for $\nrm{\Delta^3 f}_{L^2}$ 
in (\ref{est-2d-f-4}) can be decomposed in the following way: 
\begin{eqnarray} 
  &&
  \nrm{\Delta^3 f}_{2}^2 
  = \sum_{k,l=-N}^{N} I_{k,l},  \quad \mbox{with}  \nonumber 
\\  
  &&
  I_{k,l} =  L^2
    \sum_{k_1, l_1=-\infty}^\infty 
    \left( \left( \frac{2 (k+ k_1 N^*) \pi}{L} \right)^{2} 
    +  \left( \frac{2 (l+ l_1 N^*) \pi}{L} \right)^{2} \right)^6    
     \left| \hat{f}_{k+ k_1 N^*, l+ l_1 N^*} \right|^2 .  \label{est-2d-f-5}
\end{eqnarray}
In particular, we observe that 
\begin{eqnarray} 
  & & 
  \left( \left( \frac{2 (k+ k_1 N^*) \pi}{L} \right)^{2} 
    +  \left( \frac{2 (l+ l_1 N^*) \pi}{L} \right)^{2} \right)^6
   \left| \hat{f}_{k+ k_1 N^*, l+ l_1 N^*} \right|^2   \le  \frac{I_{k,l}}{L^2} ,  \nonumber 
\\
  & & 
  \quad \mbox{i.e.}  \quad 
  \left| \hat{f}_{k+ k_1 N^*, l+ l_1 N^*} \right| 
  \le  \left( \left( \frac{2 (k+ k_1 N^*) \pi}{L} \right)^{2} 
    +  \left( \frac{2 (l+ l_1 N^*) \pi}{L} \right)^{2} \right)^{-3} 
     \sqrt{\frac{I_{k,l}}{L^3}} .  
\label{est-2d-coefficient-lemma2-1}
\end{eqnarray}  
Meanwhile, the following fact is obvious 
\begin{equation} 
  \left| \lambda_{kx,(4)} + \frac{4 (k+ k_1 N^*)^2 \pi^2}{L^2}  \right|  
  \le \frac{4 (k+ k_1 N^*)^2 \pi^2}{L^2}  ,  \quad 
  \mbox{since $\lambda_{kx,(4)} \le 0$ and $| \lambda_{kx,(4)}| \le  \frac{4 k^2 \pi^2}{L^2}$} . 
  \label{est-2d-coefficient-lemma2-1-1}
\end{equation}  
Therefore, the following estimate is valid for a fixed $(k,l)$ and $(k_1, l_1) \ne (0,0)$: 
\begin{eqnarray} 
  & & \hspace{-0.75in}
  \left| \left( \lambda_{kx,(4)}  + \frac{4 (k+ k_1 N^*)^2 \pi^2}{L^2} \right) 
   \hat{f}_{k+k_1 N^*, l + l_1 N^*}  \right| 
  \le \frac{4 (k+k_1 N^*)^2 \pi^2}{L^2}  \left| \hat{f}_{k+ k_1 N^*, l+ l_1 N^*} \right|   \nonumber 
\\
  &\le& \left( \left( \frac{2 (k+ k_1 N^*) \pi}{L} \right)^{2} 
    +  \left( \frac{2 (l+ l_1 N^*) \pi}{L} \right)^{2} \right)^{-2}   
    \sqrt{\frac{I_{k,l}}{L^2}}   \nonumber 
\\  
  &\le& \left( \frac{4 (N^*)^2 \pi^2}{L^2} \left(  (|k_1| - \frac12)^2  
  + ( |l_1| - \frac12)^2 \right)  \right)^{-2}  
  \sqrt{\frac{I_{k,l}}{L^2}}  \nonumber 
\\
  &\le& \frac{1}{16}  h^4 \pi^{-4} \sqrt{\frac{I_{k,l}}{L^2}}  
  \frac{1}{\left( ( | k_1 | - \frac12)^2 + ( | l_1 | - \frac12)^2 \right)^{2} } ,  
  % \quad \mbox{with}  \quad  L_0 = \mbox{max} \left( L_x, L_y \right) , 
    \label{est-2d-coefficient-lemma2-2}  
\end{eqnarray} 
in which the grid size $h = \frac{L}{2 N+1}$ was recalled. 
Also note that we used the following estimate in the third step 
\begin{equation} 
  \left| k + k_1 N^* \right| \ge \frac{( | k_1 | - \frac12)  N^*}{2},  \quad 
  \left| l + l_1 N^* \right| \ge \frac{( | l_1 | - \frac12) N^*}{2}, \quad 
  \mbox{for} \quad  k_1 \ne 0 , \, \, l_1 \ne 0 , 
\end{equation} 
due to the fact that $| k|, | l | \le N^*/2$.  
Consequently, its substitution into (\ref{est-2d-coefficient-3}) shows that 
\begin{eqnarray} 
  & &  \hspace{-0.5in}
   \left| \sum_{k_1,l_1=-\infty\atop (k_1,l_1) \ne (0,0)}^{\infty}  \left( \lambda_{kx,(4)} 
   + \frac{4 (k+ k_1 N^*)^2 \pi^2}{L^2} \right) 
   \hat{f}_{k+k_1 N^*, l + l_1 N^*}  \right|    \nonumber 
\\
   &\le&  \sum_{k_1,l_1=-\infty\atop (k_1,l_1) \ne (0,0)}^{\infty}  
   \left|  \left( \lambda_{kx,(4)} 
   + \frac{4 (k+ k_1 N^*)^2 \pi^2}{L^2} \right) 
   \hat{f}_{k+k_1 N^*, l + l_1 N^*}  \right|  \nonumber 
\\
   &\le&  \sum_{k_1,l_1=-\infty\atop (k_1,l_1) \ne (0,0)}^{\infty}
   \frac{1}{16}  h^4 \pi^{-4} \sqrt{\frac{I_{k,l}}{L^2}}  
  \frac{1}{\left( ( | k_1 | - \frac12)^2 + ( | l_1 | - \frac12)^2 \right)^{2}  }      
   =   C^* h^4  \sqrt{\frac{I_{k,l}}{L^2}} ,
    \label{est-2d-coefficient-lemma2-3}  
	\end{eqnarray}
where
	\begin{equation}
  C^* =  \frac{1}{16} \pi^{-4} 
\sum_{k_1,l_1=-\infty\atop (k_1,l_1) \ne (0,0)}^{\infty} 
\frac{ 1 }{\left( ( | k_1 | - \frac12)^2 + ( | l_1 | - \frac12)^2 \right)^{2} } .    
	\end{equation}  
Note that the double series 
\begin{equation} 
  \sum_{k_1,l_1=-\infty\atop (k_1,l_1) \ne (0,0)}^{\infty}  
  \frac{ 1 }{\left( ( | k_1 | - \frac12)^2 + ( | l_1 | - \frac12)^2 \right)^{\beta_0} } , 
  \label{convergent-double-1}
  	\end{equation}
with $\beta_0 =2$, is convergent. In turn, we arrive at 
\begin{eqnarray} 
  & & \hspace{-0.5in}
  \sum_{k,l=-N}^{N}  
  \left| \sum_{k_1,l_1=-\infty\atop (k_1,l_1) \ne (0,0)}^{\infty}  \left( \lambda_{kx,(4)} 
   + \frac{4 (k+ k_1 N^*)^2 \pi^2}{L^2} \right) 
   \hat{f}_{k+k_1 N^*, l + l_1 N^*}  \right|^2    \nonumber 
\\
  &\le&  \sum_{k,l=-N}^{N} (C^*)^2 h^8 \cdot \frac{I_{k,l}}{L^2}   
  = \frac{(C^*)^2}{L^2} h^8 \sum_{k,l=-N}^{N}  I_{k,l}    
   = \frac{(C^*)^2}{L^2} h^8  \nrm{\Delta^3  f}^2 
     \nonumber
     \\
&=& \frac{(C^*)^2}{L^2} h^8 \nrm{f}_{H^6}^2  ,  
  \label{est-2d-coefficient-lemma2-4}  
\end{eqnarray} 
in which the decomposition (\ref{est-2d-f-5}) was used in the second to the last step. %and the definition of $\| \cdot \|_{H^}$ was used in the last step. 
Therefore, the first inequality of Lemma \ref{lem-2} %\textcolor{red}{(I corrected this reference.)} 
is proven by taking 
$C_2 = \frac{(C^*)^2}{L^2}$. The second inequality of Lemma \ref{lem-2} 
can be established in the same manner. 
This completes the proof of Lemma \ref{lem-2}.

\section{Proof of Lemma~\ref{lem: inequality}}
	\label{proof:Lemma 3}

\begin{proof} 
For any periodic grid function $f$, it has a corresponding discrete Fourier transformation: 
	\begin{eqnarray}
  f_{i,j} &=& \sum^{N}_{\ell,m=-N}
\hat{f}^N_{\ell,m} {\rm e}^{2 \pi i ( \ell x_{i} + m y_{j})/ L } , \label{def:Fourier-1} 
	\end{eqnarray}
where $x_{i} = (i - \frac12 ) h$, $y_{j} = ( j - \frac12) h$, and $\hat{f}^N_{\ell,m}$ are the coefficients. In turn, for an application of the operator $-\Delta_{h, (4)}$ to $f$, the following discrete Fourier expansion is available: 
	\begin{eqnarray}
	&&
-\Delta_h f_{i,j} = \sum^{N}_{\ell,m=-N} (\nu_\ell + \nu_m) \hat{f}^N_{\ell,m}  {\rm e}^{2 \pi i  ( \ell x_{i} + m y_{j})/L}  ,   \label{inequality-1-1} 
\\	
	  &&
-\Delta_{h, (4)} f_{i,j} = \sum^{N}_{\ell,m=-N} \Lambda_{\ell, m} \hat{f}^N_{\ell,m}  {\rm e}^{2 \pi i  ( \ell x_{i} + m y_{j})/L}  ,   \label{inequality-1-2} 
\\
  &&
   \mbox{with} \, \, \nu_k = \frac{4\sin^2 {\frac{\ell\pi h}{L}}}{h^2}  ,  \, \, 
   \mu_k =  \nu_k^2 + \frac{h^2}{12} \nu_k^2 ,  \, \,    
   \Lambda_{\ell, m} = \mu_\ell + \mu_m  .   \label{inequality-1-3}    
	\end{eqnarray}
For any $f \in \mathring{\cal V}_{\rm per}$, we observe that $hat{f}^N_{\ell,m} =0$, due to the fact that $\overline{f}=0$. In turn, a similar expansion could be derived for $ (-\Delta_{h, (4)})^{-1} f$: 
	\begin{eqnarray}
	  &&
( -\Delta_{h, (4)} )^{-1} f_{i,j} = \sum_{\ell,m \ne (0,0)} \Lambda_{\ell, m}^{-1} \hat{f}^N_{\ell,m}  {\rm e}^{2 \pi i  ( \ell x_{i} + m y_{j})/L}  .   \label{inequality-2}  
	\end{eqnarray}
	
On the other hand, the following identities could be derived, based on the orthonormal property of the Fourier basis function: 
        \begin{eqnarray} 
          &&
  \| \| f \|_{-1,h}^2 = ( f , ( -\Delta_{h, (4)} )^{-1} f )_2 
  = L^2 \sum_{\ell,m \ne (0,0)}  \Lambda_{\ell, m}^{-1} | \hat{f}^N_{\ell,m} |^2 , 
  \label{inequality-3-1} 
\\ 
     &&
   \| \nabla_{h, (4)} f \|_2^2 = ( f , -\Delta_{h, (4)}  f )_2 
  = L^2 \sum_{\ell,m \ne (0,0)}  \Lambda_{\ell, m} | \hat{f}^N_{\ell,m} |^2 , 
  \label{inequality-3-2}   
 \end{eqnarray}  
for any $f \in \mathring{\cal V}_{\rm per}$. Meanwhile, an application of Parseval equality to~\eqref{def:Fourier-1} implies that 
        \begin{eqnarray} 
          &&
  \| \| f \|^2 = L^2 \sum_{\ell,m \ne (0,0)}  | \hat{f}^N_{\ell,m} |^2 . 
  \label{inequality-4} 
	\end{eqnarray}
An application of discrete Cauchy-Schwarz inequality leads to 
    \begin{eqnarray} 
     \left( \sum_{\ell,m \ne (0,0)}  | \hat{f}^N_{\ell,m} |^2 \right)^2 
     \le \left( \sum_{\ell,m \ne (0,0)}  \Lambda_{\ell, m}^{-1} | \hat{f}^N_{\ell,m} |^2 \right) 
     \cdot \left( \sum_{\ell,m \ne (0,0)}  \Lambda_{\ell, m} | \hat{f}^N_{\ell,m} |^2 \right) , 
     \label{inequality-5}    
     \end{eqnarray}    
which is equivalent to $\| f \|_2^4 \le \| f \|_{-1,h}^2 \cdot \| \nabla_{h, (4)} f \|_2^2$, for any  $f \in \mathring{\cal V}_{\rm per}$. This completes the proof of~\eqref{inequality-0-1}. 

  The proof of~\eqref{inequality-0-2} follows a similar argument. By making a comparison between~\eqref{inequality-1-1} and \eqref{inequality-1-2}, combined with an application of Parseval equality, we get     
  	\begin{eqnarray}
	&&
  \| \Delta_h f \|_2^2 =  L^2 \sum^{N}_{\ell,m=-N} (\nu_\ell + \nu_m)^2 | \hat{f}^N_{\ell,m} |^2  ,   \label{inequality-6-1} 
\\	
	  &&
 \| \Delta_{h, (4)} f \|_2 = L^2 \sum^{N}_{\ell,m=-N} \Lambda_{\ell, m}^2 | \hat{f}^N_{\ell,m} |^2 .   \label{inequality-6-2} 
     \end{eqnarray} 
Therefore, \eqref{inequality-0-2} becomes a direct consequence of the following fact: 
\begin{eqnarray} 
  | \nu_\ell + \nu_m | = \nu_\ell + \nu_m \le  \Lambda_{\ell, m} = | \Lambda_{\ell, m} | . 
  \label{inequality-7} 
     \end{eqnarray}  
The proof of Lemma~\ref{lem: inequality} is finished.  
\end{proof}

	\section*{Acknowledgements} 
This work is supported in part by the grants NSF DMS-1418689 (C.~Wang), NSF DMS-1418692 and NSF DMS-1719854 (S.~Wise). %NSFC 11271281 (C.~Wang), 	

	\bibliographystyle{plain}
	\bibliography{draft1.bib}

	\end{document}